\documentclass[preprint]{elsarticle}

\usepackage{textcomp}

\usepackage[T1]{fontenc}
\usepackage[english]{babel}
\usepackage{amsmath,amsfonts,amssymb,amsthm}
\usepackage{mathtools}
\usepackage{graphicx}
\usepackage{url}
\usepackage{hyperref}
\usepackage{cleveref}
\usepackage{color}
\usepackage{appendix}  
\usepackage[margin=1in]{geometry}
\usepackage{float} 
\usepackage{soul} 

\usepackage{algorithm,algpseudocode}

\newcommand{\bit}{\begin{itemize}}
\newcommand{\eit}{\end{itemize}}
\newtheorem{theorem}{Theorem}

\newtheorem{definition}{Definition}

\newtheorem{remark}{Remark}

\newcommand{\Real}{\mathbb{R}}

\DeclareMathOperator{\rd}{\text{\upshape{d}}}

\DeclareMathOperator{\unvec}{unvec}

\newcommand{\bzero}{\mathbf{0}}
\renewcommand{\a}{\mathbf{a}}
\renewcommand{\b}{\mathbf{b}}

\newcommand{\f}{\mathbf{f}}
\newcommand{\g}{\mathbf{g}}
\newcommand{\h}{\mathbf{h}}

\renewcommand{\u}{\mathbf{u}}
\renewcommand{\v}{\mathbf{v}}
\newcommand{\w}{\mathbf{w}}
\newcommand{\x}{\mathbf{x}}
\newcommand{\y}{\mathbf{y}}
\newcommand{\z}{\mathbf{z}}

\newcommand{\A}{\mathbf{A}}
\newcommand{\B}{\mathbf{B}}
\newcommand{\C}{\mathbf{C}}
\newcommand{\F}{\mathbf{F}}

\newcommand{\I}{\mathbf{I}}

\newcommand{\M}{\mathbf{M}}

\renewcommand{\P}{\mathbf{P}}
\newcommand{\Q}{\mathbf{Q}}
\newcommand{\R}{\mathbf{R}}
\renewcommand{\S}{\mathbf{S}}
\newcommand{\T}{\mathbf{T}}
\newcommand{\U}{\mathbf{U}}
\newcommand{\V}{\mathbf{V}}
\newcommand{\W}{\mathbf{W}}
\newcommand{\X}{\mathbf{X}}
\newcommand{\Y}{\mathbf{Y}}

\newcommand{\enn}{{\rm n}}

\newcommand{\cE}{\mathcal{E}}
\newcommand{\cH}{\mathcal{H}}

\newcommand{\cL}{\mathcal{L}}

\newcommand{\tW}{\widetilde{\W}}

\newcommand{\tw}{\widetilde{\w}}
\newcommand{\tv}{\widetilde{\v}}

\usepackage[textsize=tiny]{todonotes}

\newcommand*\circled[1]{\tikz[baseline=(char.base)]{
\node[shape=circle, draw, inner sep=0.6pt] (char) {#1};}
}
\newcommand\kronF[2]{#1^{\circled{\tiny{#2}}}}

\journal{Computer Methods in Applied Mechanics and Engineering}

\begin{document}

\begin{frontmatter}

\title{Scalable Computation of Energy Functions for Nonlinear Balanced Truncation}
\author[UCSD]{Boris Kramer}

\affiliation[UCSD]{
organization={Department of Mechanical and Aerospace Engineering, University of California San Diego},
city={La Jolla},
postcode={92093-0411}, 
state={CA},
country={USA}
}

\author[VT]{Serkan Gugercin}
\author[VT]{Jeff Borggaard}
\author[VT]{Linus Balicki}

\affiliation[VT]{
organization={Department of Mathematics},
city={Blacksburg},
postcode={24061}, 
state={VA},
country={USA}
}

\begin{abstract}
Nonlinear balanced truncation is a model order reduction technique that reduces the dimension of nonlinear systems in a manner that accounts for either open- or closed-loop observability and controllability aspects of the system. A computational challenges that has so far prevented its deployment on large-scale systems is that the energy functions required for characterization of controllability and observability are solutions of various high-dimensional Hamilton-Jacobi-(Bellman) equations, which are computationally intractable in high dimensions.
This work proposes a unifying and scalable approach to this challenge by considering a Taylor-series-based approximation to solve a class of parametrized Hamilton-Jacobi-Bellman equations that are at the core of nonlinear balancing. The value of a formulation parameter provides either open-loop balancing or a variety of closed-loop balancing options. To solve for the coefficients of Taylor-series approximations to the energy functions, the presented method derives a linear tensor system and heavily utilizes it to numerically solve structured linear systems with billions of unknowns.
The strength and scalability of the algorithm is demonstrated on two semi-discretized partial differential equations, namely the Burgers  and the Kuramoto-Sivashinsky equations. 
\end{abstract}


\end{frontmatter}

\section{Introduction} \label{sec:intro}
Simulation of large-scale nonlinear dynamical systems can be time-consuming and resource-intensive.  It is a frequent bottleneck when these simulations are used for real-time, model-based control. Reduced-order models (ROMs) provide an attractive solution to this problem by approximating dynamical systems (and their relevant system-theoretic properties) in a much lower dimensional state space; see, e.g., \cite{AntBG20, antoulas05, BCOW2017morBook, Zho96, SvdVR08,BGW15surveyMOR} for a general overview. These ROMs then allow for the design of low-dimensional controllers and filters. 

Balanced truncation model reduction, pioneered by Moore~\cite{moore81principal} and
Mullis and Roberts~\cite{mullis1976synthesis}
for \textit{linear} time-invariant (LTI) systems, provides an elegant approach to model reduction for open-loop settings. The approach uses the controllability and observability energies of a system to determine those states that have the most relative importance. If a state requires a large amount of input energy to be reached and also has minimal effect on the output, then the reduced model would neglect that state without a significant impact on the input-output behavior of the system. Extensions of this concept to closed-loop LTI systems led to the LQG balancing \cite{verriest1981suboptimal,jonckheere1983LQGbalancing} and $\mathcal{H}_\infty$ balancing concepts~\cite{glover91Hinftybalancing}. 
Several variants for different formulations of the LTI system and different reduced-model outcomes exist, such as stochastic balancing~\cite{desai1984stochasticBal,green1988balancedstochastic}, bounded real balancing~\cite{opdenacker1988contraction}, positive real balancing~\cite{desai1984stochasticBal,ober1991balanced}, frequency-weighted balancing~\cite{enns1984model}; see also the surveys~\cite{gugercin2004survey,benner2017chapter}. The interest in balancing methods for LTI systems in the 1980s stimulated research in computational methods for solving large-scale Lyapunov or Riccati-type algebraic matrix equations that yielded low-rank solvers \cite{benner2013numerical,simoncini2016computational} and doubling methods~\cite{li2013solving} that can solve these matrix equations for millions of states.

For \textit{nonlinear} large-scale systems, the theory of balanced truncation is largely developed, yet computationally scalable approaches and efficient ROM development remain open problems. 
The theoretical foundation for balanced truncation of nonlinear open-loop systems proposed by Scherpen~\cite{scherpen1993balancing} defines input and output energy functions and shows that they can be computed as solutions to Hamilton-Jacobi ({HJ}) partial differential equations (PDEs). Theoretical extensions of~\cite{scherpen1993balancing} to the closed-loop setting, such as Hamilton-Jacobi-Bellman (HJB) balancing~\cite{scherpen1994normCoprime}  and $\mathcal{H}_\infty$ balancing~\cite{scherpen1996hinfty_balancing} have also been proposed for nonlinear systems. 
As in the linear case, there are also structure-preserving balancing such as dissipativity-preserving~\cite{ionescu2010dissipativity} and positive-real balancing~\cite{ionescu2009positive} for nonlinear systems.
In all of these approaches though, the main challenge is the computational bottleneck of solving HJB PDEs, which is intractable for systems of even modest order. 
%

%
Fujimoto and Tsubakino~\cite{fujimoto2008TaylorSeriesBT} use Taylor series approximations of the controllability and observability energy functions, as well as the singular value functions, to solve the HJ equations (an idea that goes back to \cite{Albrekht1961,lukes1969optimal}). Their subsequent balanced ROMs are illustrated on a four-dimensional ordinary differential equation (ODE). Building on~\cite{fujimoto2008TaylorSeriesBT}, the authors in \cite{sahyoun2013reduced} consider the special class of quadratic models and use Taylor series expansion to solve the HJB equations and balancing transformations.  However, no numerical examples are given. 
In these existing Taylor-series-based approaches for nonlinear balancing, an explicit structure to obtain the coefficients of the energy function and a corresponding scalable numerical algorithm are missing.

An alternative approach to solving the HJB equations is to use max-plus methods~\cite{mceneaney2007maxplusHJB}, which can provide significant computational savings, but require highly nontrivial customization to dynamical systems forms; those methods are hence less generalizable when compared to Taylor-series based methods.
Instead of focusing on the energy functions, approximate balanced truncation methods for nonlinear systems have been proposed that focus on algebraic Gramians, which is appealing from a computational perspective. For instance~\cite{verriest2006algebraicGramiansNLBT} derives such Gramians as the solution of Lyapunov-type equations, \cite{al1994new} proposed algebraic Gramians for bilinear systems, and \cite{condon2005nonlinear} combined Carleman bilinearization with that balancing method to approximately balance weakly nonlinear systems. A computationally efficient framework via truncated algebraic Gramians for quadratic-bilinear systems has been proposed in~\cite{bennerGoyal2017BT_quadBilinear}, which also applies to more general nonlinear systems via lifting, see~\cite{KW2019_balanced_truncation_lifted_QB}.
Empirical Gramians for nonlinear systems~\cite{lall2002subspace} can also be used, yet their computation requires as many simulations of the full system as there are inputs and outputs: each simulation uses an impulse disturbance per input channel while setting the other inputs to zero. 
Newman and Krishnaprasad~\cite{newman2000MC_BT} show that the controllability energy function is related to the stationary density $p_\infty$ of a Markov process. They suggest to solve the corresponding Fokker-Planck equations for $p_\infty$ instead and further show that for a particular class of systems with Hamiltonian structure, this relationship is exact. However, this shifts the computational burden to solving the Fokker-Planck equations. Hence, their test model is a two-dimensional ODE.
A method for approximating the nonlinear balanced truncation reduction map via reproducing kernel Hilbert spaces (a machine learning-based, data-driven technique) has been proposed in~\cite{bouvrie2017kernel}, and applied to two- and seven-dimensional ODE examples. 
As evident from the literature, scalable and computationally efficient nonlinear balanced truncation approaches based on the energy functions that work for both closed- and open-loop systems are currently lacking for medium- and large-scale systems.

There are three main contributions of this paper.
First, we propose a unifying Taylor series-based solution to a class of parametrized HJB equations that yield the open- \textit{and} closed-loop nonlinear balancing energy functions. Depending on the value of a parameter $\gamma$, we either obtain open-loop ($\gamma=1$) or otherwise a family of closed-loop energy functions. 
Second, we derive the explicit tensor structure for the coefficients of the Taylor expansion. Our numerical methods subsequently exploit this structure to allow scalability up to thousands of state variables.
Third, we provide a scalability and solvability analysis, large-scale numerical experiments,  and open-access software for all algorithms and examples in the \texttt{NLbalancing} repository~\cite{NLbalancing}.

Our computational solution to the parametrized HJB equations assumes a quadratic dynamical system. The motivation for this is threefold: 
First, from a computational perspective, solving the $\cH_\infty$ equations (see equation~\eqref{eq:HinftyenergyFunctions1} below) for a general nonlinear function remains very challenging beyond even two degrees of freedom. No general approaches exist that would lend themselves to the scalability required for model reduction. 
Second, many models in fluid mechanics are quadratic, such as the Navier-Stokes, Burgers, Kuramoto-Sivashinsky, Fisher-KPP equations, and the shallow-water equations. These equations model nonlinear fluid flow, include phenomena such as travelling waves and shocks, and serve as numerical test beds for more complex engineering models. 
Third, focusing on quadratic drift is not very restrictive: many nonlinear dynamical systems can be transformed into quadratic nonlinearities with additional bilinear terms via variable transformations and the introduction of auxiliary variables \cite{mccormick1976computability, kerner1981universal,savageau1987recasting,gu2011qlmor,bennerBreiten2015twoSided,KW18nonlinearMORliftingPOD,SKHW2020_learning_ROMs_combustor,guillot2019taylor,QKPW2020_lift_and_learn}. Explicit algorithms to find such quadratizations can be found in \cite{lifeware2,lifeware1,Bychkov2021,ByIsPoKr_Quadratization_2023}.

This paper is organized as follows. Section~\ref{sec:NLBal} reviews energy functions for nonlinear systems, their associated HJB equations and resulting controllers. Section~\ref{sec:computingEF} presents the proposed computational framework, including an analysis thereof, for efficiently solving high-dimensional parametrized HJB equations. 
Section~\ref{sec:numerics} presents numerical results for two semi-discretized PDEs: Burgers equation and the the Kuramoto-Sivashinsky equation. Section~\ref{sec:conclusions} offers conclusions and an outlook toward future work.

\section{Energy functions for nonlinear balancing} \label{sec:NLBal}
Nonlinear balanced truncation proceeds in two stages: First, define an energy function (e.g., controllability and observability), and compute those efficiently. Second, find a variable transformation that ``diagonalizes" these energy functions. In the linear case, these energy functions are quadratic, and diagonalization means that the \textit{linear} transformation is chosen so that the observability and controllability Gramians are diagonal. The nonlinear case is more involved, as both the energy functions and transformations become nonlinear. 
Section~\ref{ss:energyFcts} reviews background material on energy functions, the associated HJB equations, and the resulting controllers. Section~\ref{ss:NL-Bal} defines the nonlinear observability and controllability energy functions for the open-loop setting. Section~\ref{ss:HinftyBal} defines nonlinear $\cH_\infty$-type energy functions for closed-loop systems. Special cases are provided in the appendix for the interested reader. Section~\ref{ss:HJBbal} reviews energy functions from the HJB-balancing framework; Section~\ref{ss:LQGbal} defines extensions of the LQG energy functions for nonlinear systems.

\subsection{Energy functions, Hamilton-Jacobi-Bellman equations, and optimal controllers}\label{ss:energyFcts}
The balanced truncation problem is tightly connected to the optimal control problem, which we illustrate in this background section. We consider the finite-dimensional nonlinear dynamical system
\begin{align}
\dot{\x}(t)  = \f(\x(t))  + \g(\x(t)) \u(t), \qquad \y(t) = \h(\x(t)),\label{eq:FOMNL}
\end{align}
where $t$ denotes time, $\x(t) \in \Real^{n}$ is the state, 
$\u(t) \in \Real^m$ is a time-dependent input vector, 
$\g:\Real^n\mapsto\Real^{n\times m}$  encodes the actuation mechanism,
$\y(t) \in \Real^p$ is the output vector measured by the function $\h:\Real^n \mapsto \Real^p$, the nonlinear drift term is $\f:\Real^n\mapsto\Real^n$, and $\x =\bzero$ is an isolated equilibrium for $\u=\bzero$.
Consider the energy function
\begin{equation}
    \widehat{\cE}(\x_0,\u) =  \frac{1}{2} \int_{0}^\infty Q(\x(t)) + R(\u(t)) \text{d} t, \quad \x(0)=\x_0, \label{eq:energy_general}
\end{equation}
where the state and control penalty functions $Q(\x(t))$ and $R(\x(t))$ are non-negative. The goal of a nonlinear feedback control is to find a control $\u^*=\u^*(\x)$ that minimizes the energy (or cost) in \eqref{eq:energy_general}, i.e.,
\begin{equation} \label{eq:optimalcontrol}
\u^* := \arg \min_\u \widehat{\cE}(\x_0,\u) \quad \Rightarrow \quad \cE(\x_0) := \widehat{\cE}(\x_0,\u^*).
\end{equation}
A necessary (and in fact also sufficient~\cite{almubarak2019infinite}) condition is that the minimizer $\u^*$ must satisfy the HJB equation
\begin{equation} \label{eq:HJBgeneral}
    0 = \min_\u \left \{ Q(\x) + R(\u) + \frac{\partial \cE(\x)}{\partial \x} \left [\f(\x) + \g(\x) \u \right ] \right \}.
\end{equation}
With a few additional assumptions and a quadratic cost function, we can obtain a closed form feedback solution to the nonlinear control problem~\eqref{eq:optimalcontrol}.
\begin{theorem}[\hspace{1sp}\cite{lukes1969optimal}] \label{thm:Lukes}
Consider the nonlinear dynamical system~\eqref{eq:FOMNL} with the quadratic cost function
\begin{equation}
    \widehat{\cE}(\x_0,\u) =  \frac{1}{2} \int_{0}^\infty \x(t)^\top \Q \x(t) + \u(t)^\top \R \u(t) {\rm{d}} t.
\end{equation}
Let the following assumptions hold: (1) there exists a neighborhood $\Omega\subseteq \Real^n$ of the origin where $\f\in C^2(\Omega)$; $\f(\bzero) = \bzero$; (2) the pair $\left ( \frac{\partial \f}{\partial \x} (\bzero), \g(\bzero) \right )$ is stabilizable; (3) the system~\eqref{eq:FOMNL} is stabilizable on $\Omega$, so there exists a stabilizing controller such that the closed-loop system is asymptotically stable on $\Omega$. 
Then there exists a unique continuously differentiable minimizer for the optimal feedback control $\u^*(\x)$ that solves the HJB equation~\eqref{eq:HJBgeneral}, and this minimizer is given by
\begin{equation} \label{eq:ustar}
    \u^*(\x) = -\R^{-1} \g(\x)^\top \frac{\partial^\top \cE(\x)}{\partial \x}.
\end{equation}
Moreover, if $\f(\x)$ is analytic, so are $\u^*(\x)$ and $\cE(\x)$.
\end{theorem}
Since the terms inside the integral of the energy function are quadratic in $\x^*$ and $\u^*$, it can be shown directly that this energy function satisfies the Lyapunov conditions (i.e., $\cE(\x) >0$ and $\frac{\text{d}\cE}{\text{d}t}(\x) <0$ for all $\x \in \Omega \backslash \bzero$). 
Specifically, under the assumptions of Theorem~\ref{thm:Lukes}, the energy function $\cE(\x)$ in \eqref{eq:optimalcontrol} is a Lyapunov function for the nonlinear control-affine system~\eqref{eq:FOMNL}. In other words, with the controller $\u^*$ from Theorem~\ref{thm:Lukes}, the origin of the closed-loop system $\dot{\x}(t) = \f(\x(t))  + \g(\x(t)) \u^*(t)$ is asymptotically stable, see, e.g., \cite{almubarak2019infinite}.

In sum, once the HJB equation is solved for the energy function, an optimal control can be found via~\eqref{eq:ustar}.  The next sections detail different approaches to devising an energy function via HJ(B) equations.

\subsection{Nonlinear observability and controllability energy functions} \label{ss:NL-Bal}
Similar to the linear case~\cite{moore81principal}, for an {{\textit{asymptotically stable}}} nonlinear system~\eqref{eq:FOMNL} controllability and observability energy functions can be defined~\cite{scherpen1993balancing} as
\begin{align} \label{eq:energyFunctions1}
\cE_c(\x_0) & :=\min_{\substack{\u \in L_{2}(-\infty, 0] \\ \x(-\infty) = \bzero  \\ \x(0) = \x_0}} \ \frac{1}{2} \int_{-\infty}^{0} \Vert \u(t) \Vert^2 \text{d}t \\
\cE_o(\x_0) & :=\frac{1}{2} \int_{0}^{\infty} \Vert \y(t) \Vert^2 \text{d}t, \label{eq:energyFunctions2}
\end{align}
where $\mathcal{E}_c(\x_0)$ quantifies the minimum amount of energy required to steer the system from $\x(-\infty)=\bzero$ to $\x(0)=\x_0$, and $\mathcal{E}_o(\x_0)$ quantifies the output energy generated by the nonzero initial condition $\x_0$ and $\u(t)\equiv\bzero$.
The energy functions in \eqref{eq:energyFunctions1}--\eqref{eq:energyFunctions2} are solutions to Hamilton-Jacobi equations:
\begin{align} 
0 = & \frac{\partial \cE_c(\x)}{\partial \x } \f(\x) +\frac{1}{2}\frac{\partial \cE_c(\x)}{\partial \x} \g(\x) \g(\x)^\top \frac{\partial^\top \cE_c(\x)}{\partial \x} \label{eq:regularcase2}, \\[1ex]
0 = & \frac{\partial \cE_o(\x)}{\partial \x} \f(\x) + \frac{1}{2} \h(\x)^\top  \h(\x) \label{eq:regularcase1} .
\end{align}
Conditions for the energy functions to exist (i.e., to be finite) are that $\f$ is {{asymptotically stable}} in a neighborhood of the origin, and that $-\left (\f(\x) + \g(\x)\g(\x)^\top \frac{\partial^\top \cE_c(\x)}{\partial \x} \right )$ is asymptotically stable in a neighborhood of the origin~\cite{scherpen1993balancing}. For {asymptotically stable} {and minimal} linear dynamical systems, $\dot{\x} = \A\x + \B\u, \ \y = \C\x$, the controllability and observability functions are quadratic in the state, i.e.,
\begin{equation}
\cE_c(\x_0) = \frac{1}{2} \x_0^\top \P^{-1} \x_0, ~\quad \cE_o(\x_0) = \frac{1}{2} \x_0^\top \Q \x_0,
\end{equation}
where  $\P$ and $\Q$ are the unique positive definite solutions to the Lyapunov equations
\begin{equation} \label{eq:lyap}
\A \P + \P \A^\top + \B \B^\top = \bzero,~\quad
\A^\top \Q + \Q \A + \C^\top \C = \bzero,
\end{equation}
which can be derived from equations~\eqref{eq:regularcase2}--\eqref{eq:regularcase1}. The following section focuses on the closed-loop setting, which can also address unstable systems.

\subsection{\texorpdfstring{$\mathcal{H}_\infty$}{H-infinity} balancing for nonlinear systems} \label{ss:HinftyBal}
The development of a unifying $\mathcal{H}_\infty$ control framework for nonlinear systems creates a direct need for $\mathcal{H}_\infty$ balancing~\cite{scherpen1996hinfty_balancing}, which we briefly review herein. 
%
\begin{definition}\cite[Def. 5.1]{scherpen1996hinfty_balancing}
For the nonlinear dynamical system in 
\eqref{eq:FOMNL}, the $\mathcal{H}_\infty$ past energy in the state $\x_0$ is defined, for $0<\gamma$, $\gamma\neq 1$, as
\begin{equation} \label{eq:HinftyenergyFunctions1}
\cE_\gamma^{-}(\x_0)  :=\min_{\substack{\u \in L_{2}(-\infty, 0] \\ \x(-\infty) = \bzero, \\  \x(0) = \x_0}} \ \frac{1}{2} \int\displaylimits_{-\infty}^{0} (1-\gamma^{-2})\Vert \y(t) \Vert^2  +  \Vert \u(t) \Vert^2 {\rm{d}}t.
\end{equation}
Furthermore, the  $\mathcal{H}_\infty$ future energy in the state $\x_0$ is defined as 
\begin{equation} \label{eq:HinftyenergyFunctions2}
\cE_\gamma^{+}(\x_0)  :=\max_{\substack{\u \in L_{2}[0,\infty) \\ \x(0) = \x_0, \\  \x(\infty) = \bzero}} \ \frac{1}{2} \int\displaylimits_{0}^{\infty} \Vert \y(t) \Vert^2  +  \frac{\Vert \u(t) \Vert^2}{1-\gamma^{-2}} {\rm{d}}t,\quad {\rm{and}} \quad 
\cE_\gamma^{+}(\x_0)  :=\min_{\substack{\u \in L_{2}[0,\infty) \\ \x(0) = \x_0 , \\ \x(\infty) = \bzero}} \ \frac{1}{2} \int\displaylimits_{0}^{\infty} \Vert \y(t) \Vert^2  +  
\frac{\Vert \u(t) \Vert^2}{1-\gamma^{-2}} {\rm{d}}t
\end{equation}	
for $0<\gamma<1$ (left equation) and  for $\gamma>1$ (right equation). 
\end{definition}
As we see next, these energy functions take a special and well-known form when defined for an LTI system,  a result that goes back to~\cite[Prop. 4.8]{glover91Hinftybalancing}.

\begin{theorem}	\label{thm:HinftyARE} \cite[Thm. 3.4]{scherpen1996hinfty_balancing}
Let $\gamma_0$ denote the smallest $\tilde{\gamma}$ such that a stabilizing controller exists for which the $\mathcal{H}_\infty$ norm of the closed-loop system is less than $\tilde{\gamma}$.
Let $\gamma > \gamma_0\geq 0$. For an LTI system, the energy functions are quadratic, $\cE^{-}_\gamma(\x) = \frac{1}{2} \x^\top \Y_\infty^{-1} \x  $ and $\cE^{+}_\gamma(\x)  =\frac{1}{2} \x^\top  \X_\infty \x$, where $\Y_\infty, \ \X_\infty$ are the stabilizing symmetric positive definite solutions to the $\mathcal{H}_\infty$ AREs
\begin{align}
    \A \Y_\infty + \Y_\infty \A^\top + \B \B^\top - (1-\gamma^{-2}) \Y_\infty \C^\top \C \Y_\infty & = \bzero, \label{eq:HinftyRiccati1}\\
    \A^\top \X_\infty + \X_\infty \A + \C^\top \C - (1-\gamma^{-2}) \X_\infty \B \B^\top \X_\infty & = \bzero. \label{eq:HinftyRiccati2}
\end{align}
\end{theorem}
\begin{remark}
While the exact computation of $\gamma_0$ is an open problem, we highlight two strategies to find the constant $\gamma_0$. The first strategy, drawing from \cite[Sec. 3]{scherpen1996hinfty_balancing}, is to solve the two $\mathcal{H}_\infty$ AREs while ensuring that $\lambda_\text{max}(\X_\infty \Y_\infty)<\gamma^2$. This  guarantees  a controller that solves the associated control problem. One continues to decrease $\gamma$ and finds the limit $\gamma_0$. However, this can be rather cumbersome. 
The second strategy is based on the fact that the constant $\gamma_0$ is the infimum of the $\cH_\infty$ norms of all closed-loop system transfer functions (see~\cite[Sec.3]{scherpen1996hinfty_balancing} for further details). Then, one can use the upper and lower bound in~\cite[Rm 4.9]{glover91Hinftybalancing}:
\begin{equation}
    0 < \hat{\gamma}_0 - 1 \leq \gamma_0 \leq \hat{\gamma}_0 + 1,
\end{equation}
where the $\hat{\gamma}_0 = \sqrt{1 + \lambda_\text{max} ( \X_\infty\Y_\infty) } >1$ and $\X_\infty,\Y_\infty$ are the solutions to the standard control and filter ARE's, i.e., \eqref{eq:HinftyRiccati1} and \eqref{eq:HinftyRiccati2} with $1-\gamma^{-2} = 1$. We follow this strategy to lower-bound $\gamma_0$, and subsequently assume $\gamma > \gamma_0 \geq \hat{\gamma}_0 - 1$.
\end{remark}


Having defined the energy functions, the next theorem shows that they can  be computed, as before, via the solution of HJB equations. 
\begin{theorem}\cite[Thm 5.2]{scherpen1996hinfty_balancing}
Assume that the HJB equation

\begin{equation} \label{eq:HJB-NLHinfty2}
 0  = \frac{\partial \cE_\gamma^{-}(\x)}{\partial \x} \f(\x) + \frac{1}{2}  \frac{\partial \cE_\gamma^{-}(\x)}{\partial \x} \g(\x) \g(\x)^\top \frac{\partial^\top \cE_\gamma^{-}(\x)}{\partial \x} - \frac{1}{2} (1-\gamma^{-2}) \h(\x)^\top  \h(\x)
\end{equation}
has a solution with $\cE_\gamma^{-}(\bzero) = 0$ that also satisfies  
$
    - \left ( \f(\x) +\g(\x) \g(\x)^\top \frac{\partial^\top \cE_\gamma^{-}(\x)}{\partial \x} \right ) 
$
is asymptotically stable. Then this solution is the past energy function $\cE_\gamma^{-}(\x)$ from \eqref{eq:HinftyenergyFunctions1}.
Furthermore, assume that the HJB equation
\begin{equation} \label{eq:HJB-NLHinfty1}
0  = \frac{\partial \cE_\gamma^{+}(\x)}{\partial \x} \f(\x) - \frac{1}{2} (1-\gamma^{-2}) \frac{\partial \cE_\gamma^{+}(\x)}{\partial \x} \g(\x) \g(\x)^\top \frac{\partial^\top \cE_\gamma^{+}(\x)}{\partial \x} \\
+ \frac{1}{2}\h(\x)^\top \h(\x) 
\end{equation}
has a solution with $\cE_\gamma^{+}(\bzero) = 0$ which satisfies that
$
    \f(\x) - (1-\gamma^{-2})\g(\x) \g(\x)^\top \frac{\partial^\top \cE_\gamma^{+}(\x)}{\partial \x} 
$
is asymptotically stable. Then this solution is the future energy function $\cE_\gamma^{+}(\x)$ from \eqref{eq:HinftyenergyFunctions2}. 
\end{theorem}

Note, that for $(1-\gamma^{-2}) = -1$, i.e., for $\gamma = \sqrt{1/2}$,  \eqref{eq:HJB-NLHinfty2} and \eqref{eq:HJB-NLHinfty1} are the same, and so are their solutions:
$
\cE_{\gamma = \frac{1}{\sqrt{2}}}^{-}(\x) = \cE_{\gamma = \frac{1}{\sqrt{2}}}^{+}(\x).
$
We can also see this by dividing equation~\eqref{eq:HinftyenergyFunctions1} with $(1-\gamma^{-2})$ and reversing time. 
%
%
%
%
Moreover, under the assumption that the energy functions exist and are smooth, it has been shown in \cite[Thm 5.5. \& 5.7.]{scherpen1996hinfty_balancing} that the open-loop nonlinear energy functions $\cE_c(\x)$ and $\cE_o(\x)$  from~\eqref{eq:energyFunctions1}--\eqref{eq:energyFunctions2} can be obtained in the limit $\gamma\rightarrow 1$ of the $\mathcal{H}_\infty$ energy functions
$\cE_\gamma^{-}(\x)$ and $\cE_\gamma^{+}(\x)$, i.e., 
\begin{equation} \label{eq:limitto1}
\lim_{\gamma\rightarrow 1} \cE_\gamma^{-}(\x) = \cE_c(\x),
\quad 
\lim_{\gamma\rightarrow 1} \cE_\gamma^{+}(\x) = \cE_o(\x).
\end{equation}
%
{The two limits in \eqref{eq:limitto1} can be intuited: First, $\lim_{\gamma\rightarrow 1} \cE_\gamma^{-}(\x)$, 
$\gamma \rightarrow 1$ implies that
the term $(1-\gamma^{-2}) \rightarrow 0$, so the weighting of the output norm $\Vert \y\Vert$ vanishes in the integral \eqref{eq:HinftyenergyFunctions1}, resulting in the open loop energy function $\cE_c(\x)$.
Second, $\lim_{\gamma\rightarrow 1} \cE_\gamma^{+}(\x)$ in \eqref{eq:limitto1}, we need to consider the limit from above and below  since there are different definitions for the energy functions, see~\eqref{eq:HinftyenergyFunctions2}. In  $\lim_{\gamma\rightarrow 1^-} \cE_\gamma^{+}(\x)$, the factor $1/(1-\gamma^{-2})$ is negative, so to maximize the integral in \eqref{eq:HinftyenergyFunctions2}, we set $\Vert \u \Vert =0$, yielding the open loop energy function $\cE_o(\x)$. In $\lim_{\gamma\rightarrow 1^+} \cE_\gamma^{+}(\x)$, the factor $1/(1-\gamma^{-2})$ is positive, but this time we minimize the integral in \eqref{eq:HinftyenergyFunctions2}, right equation, and again $\Vert \u \Vert =0$ achieves that. Hence, the limit is also  $\cE_o(\x)$.
}

Since the HJB and standard balancing energy functions can be obtained as limits of the $\cH_\infty$ energy functions, 
we henceforth derive the computational framework for the most general case of $\cH_\infty$ energy functions.

\section{Computing energy functions via polynomial approximations} \label{sec:computingEF}
To develop a computationally scalable approach to solving HJB equations for nonlinear balanced truncation, we focus on polynomial nonlinear systems of the form	
\begin{align}
\dot{\x}(t) & = \A \x(t)  + \sum_{k=2}^\ell \F_k \kronF{\x}{k}(t) + \B \u(t), \label{eq:FOMx}\\
\y(t) &= \C\x(t),\label{eq:FOMy}
\end{align}
where $\A \in \Real^{n\times n}$ is the linear system matrix, $\F_k \in \Real^{n\times n^k}$ represents matricized higher-order tensors, $\B \in \Real^{n\times m}$ is an input matrix, and $\C \in \Real^{p\times n}$ is a linear measurement matrix. For a compact notation, we define the $k$-term Kronecker product of $\x$ as
\begin{equation}
\kronF{\x}{k}: = \underbrace{\x \otimes \ldots \otimes \x}_{k \ \text{times}}
\in \Real^{n^k}.
\end{equation}
For $\M\in\mathbb{R}^{q\times n}$ define the \textit{k-way Lyapunov matrix} or a special \textit{Kronecker sum} \cite{benzi2017approximation} matrix as
\begin{equation}
\label{eq:dWayLyapunov}
    \cL_k(\M) := \underbrace{\M \otimes \ldots \otimes \I_n}_{k \  \text{times}} + \cdots + \underbrace{\I_n \otimes \ldots \otimes \M}_{k \ \text{times}} \in \Real^{n^{k-1}q \times n^k}.
\end{equation}
To simplify the $\cH_\infty$ gain parameter notation, we introduce
\begin{align} 
\label{eq:eta}
    \eta := 1-\gamma^{-2}
\end{align}
so that we have 
\begin{align} 
\begin{split}
    \begin{matrix} 
        0<\gamma < 1            & \Rightarrow &  -\infty  <  \eta   <  0 \\ 
        1 \leq \gamma < \infty  & \Rightarrow &  0       \leq \eta \leq 1 
    \end{matrix} \ .
\end{split}
\end{align}

As we pointed out in the introduction, many complex systems have quadratic  drift, therefore we focus our derivation on the quadratic case, i.e., $\ell=2$ in \eqref{eq:FOMx}.
Since there is now only one term, $\F_2$ in \eqref{eq:FOMx}, we drop the subscript $k$ and replace the notation $\F_k$ with $\F$ from here out.
We also note that the assumptions of linear $\B$ and $\C$ and quadratic drift are merely used in this paper to obtain a scalable computation of the energy functions. Once the energy functions are obtained, they can be used to obtain a balancing transformation or a controller for general quadratic systems.
The algorithms developed in this paper are implemented in the \texttt{NLbalancing} repository \cite{NLbalancing}.

\subsection{Polynomials in Kronecker product form\label{sec:polynomials}} 
For convenience and to ensure a unique representation of the coefficients, we impose symmetry of our coefficients in all monomial terms in the energy functions.
\begin{definition}[Symmetric Coefficients\label{def:sym}] A monomial term with real coefficients $\w_d^\top \kronF{\x}{d}$ has {\em symmetric coefficients} if it satisfies
\begin{displaymath}
  \w_d^\top\! \left(\a_1 \otimes \a_2 \otimes \cdots \otimes \a_d\right) = \w_d^\top\! \left(\a_{i_1} \otimes \a_{i_2} \otimes \cdots \otimes \a_{i_d}\right),
\end{displaymath}
where the indices $\{ i_k \}_{k=1}^d$ are any  permutation of $1, \ldots, d$.
\end{definition}
This definition generalizes the definition of symmetry from matrices to tensors.  For example, requiring $\w_2^\top (\a\otimes \b) = \w_2^\top (\b\otimes \a)$ for any $\a$ and $\b$ is equivalent to $(\a^\top\otimes\b^\top)\w_2 = (\b^\top\otimes\a^\top)\w_2$.  Hence, using $\w_2 = \text{vec}(\W_2)$, we have $\b^\top \W_2 \a = \a^\top \W_2 \b$. Since these are real scalars, this implies $\W_2 = \W_2^\top$.

We also remark that any polynomial can be uniquely written in Kronecker product form with symmetric coefficients.  For example,
\begin{equation*}
  c_1 x_1^2 + c_2 x_1x_2 + c_3 x_2^2 = [x_1\ x_2] \left[\begin{array}{cc} c_1 & \frac{1}{2} c_2 \\ \frac{1}{2} c_2 & c_3 \end{array} \right] \left[ \begin{array}{c} x_1 \\ x_2 \end{array} \right]
  = \left[ c_1\ \ \frac{1}{2} c_2\ \ \frac{1}{2} c_2\ \  c_3\right] (\x\otimes \x). 
\end{equation*}
The same set of quadratic terms would be realized with the coefficient matrices corresponding to either $[c_1\ c_2\ 0\ c_3]$ or $[c_1\ 0\ c_2\ c_3]$.  However, the requirement of symmetry leads to a unique representation. Note that all coefficient vectors which define the same polynomial can be transformed into the unique symmetric representation by applying a symmetrization function  \cite{comon2008SymmetricTensorsandSymmetricTensorRank}. Computing the symmetrization consists of summing up the coefficients of each unique monomial term, dividing by the number of times it appears, then redistributing the results. This is a generalization of the symmetrization of a matrix. For an implementation of this, see the function \texttt{kronMonomialSymmetrize.m} in \cite{KronTools}.

We will assume that each row of {the coefficient matrix $\F = \F_2$ in \eqref{eq:FOMx}} is symmetric as in Definition \ref{def:sym}, and that the polynomial representations of the energy functions and controls share this symmetric representation.  Our algorithms are designed to ensure symmetry in the computed coefficients.

\subsection{Future energy function\label{sec:FEF}}
The future energy function in~\eqref{eq:HinftyenergyFunctions2} is a solution to equation \eqref{eq:HJB-NLHinfty1}. For  quadratic dynamics, i.e., $\ell = 2$ in~\eqref{eq:FOMx}, this equation becomes
\begin{equation} \label{eq:HinftyenergyFunctions1_Kron}
 0 =  \frac{\partial \cE_\gamma^{+}(\x)}{\partial \x} \left [ \A \x + \F(\x\otimes \x) \right ] - \frac{\eta}{2}  \frac{\partial \cE_\gamma^{+}(\x)}{\partial \x} \B \B^\top \frac{\partial^\top \cE_\gamma^{+}(\x)}{\partial \x} + \frac{1}{2}(\text{vec}(\C^\top \C))^\top (\x\otimes \x).
\end{equation}
%
We assume the future energy function is represented in the form (or is approximated as)
\begin{align}
\cE_\gamma^+(\x)  
& \approx \frac{1}{2}\left ( \w_2^\top \kronF{\x}{2}  + \w_3^\top \kronF{\x}{3} + \ldots + \w_d^\top \kronF{\x}{d} \right ) \label{eq:wi_coeffs}\\
& = \frac{1}{2}\left ( \w_2^\top + \tilde{\w}_3^\top (\x)+ \ldots + \tilde{\w}_d^\top (\x) \right )\kronF{\x}{2}, \label{eq:wi_coeffs_quad}
\end{align}
for some integer $d\geq 2$, which denotes the polynomial degree, and {$\w_k \in \Real^{n^k}$}. 
Equation~\eqref{eq:wi_coeffs} has the interpretation of an energy function with the constant coefficients $\w_i$ and higher-order polynomials. On the other hand, we can think of~\eqref{eq:wi_coeffs_quad} as a quadratic expansion with state-dependent coefficients. The latter representation is often used for Gramian-based model reduction that locally linearizes the $\tilde{\w}_i (\x)$, see, e.g., \cite{al1994new,condon2005nonlinear,verriest2006algebraicGramiansNLBT,bennerGoyal2017BT_quadBilinear}, and the references therein.
We note that polynomial energy function approximations have recently been adapted to solve optimal control problems for moderately sized bilinear systems \cite{breiten2018numerical,breiten2019taylor}, quadratic drift systems \cite{breiten2019feedbackNS,borggaard2019QQR}, and polynomial drift systems \cite{borggaard2021PQR,almubarak2019infinite}.
The next result shows how the coefficients
$\{\w_i\}_{i=2}^d$ in \eqref{eq:wi_coeffs} can be efficiently computed using tensor linear algebra tools.
\begin{theorem} \label{thm:wi}
Let $\gamma>\gamma_0\geq0$ hold as in  Theorem~\ref{thm:HinftyARE} and $\eta = 1-\gamma^{-2}$ as defined in~\eqref{eq:eta}.
Let the future energy function $\cE_\gamma^+(\x)$ for the quadratic system ($\ell =2$ and $\F_2 = \F$) in \eqref{eq:FOMx} be expanded as in \eqref{eq:wi_coeffs} 
with the coefficients $\w_i$ for $i=2, 3, \ldots, d$. Then,
$\w_2 = \text{vec}(\W_2)$ where $\W_2$ is the symmetric positive definite solution to the $\mathcal{H}_\infty$ Riccati equation
    \begin{equation} \label{eq:W2}
        \bzero = \A^\top \W_2 + \W_2\A + \C^\top \C - \eta \W_2 \B \B^\top \W_2.
    \end{equation}
    For $2<k\leq d$, let $\tw_k\in \Real^{n^k}$ solve the linear system
    \begin{equation}\label{eq:th6}
        \cL_k( (\A-\eta \B \B^\top\W_2)^\top)
        \tw_k = 
        - \cL_{k-1}(\F^\top) \w_{k-1} + \frac{\eta}{4}\sum_{\substack{i,j>2i+j=k+2}} \!\!\!\!ij~{\rm{vec}}(\W_i^\top \B \B^\top \W_j).
    \end{equation}
    Then,  the coefficient vector $\w_k = {\rm{vec}}(\W_k) \in \Real^{n^k}$, for $2<k\leq d$, is obtained by the symmetrization of $\tw_k$.
\end{theorem}

\begin{proof}
We start by observing that the derivative of the polynomial expansion~\eqref{eq:wi_coeffs} with respect to $\x$  is
\begin{align}
& \frac{\partial \cE_\gamma^{+}(\x)}{\partial \x} 
 = \frac{1}{2}\left (\w_2^\top (\I \otimes \x) + \w_2^\top (\x \otimes \I) \nonumber \right.+ \w_3^\top (\I \otimes \x \otimes \x) +  \w_3^\top (\x \otimes \I \otimes \x) + \w_3^\top (\x \otimes \x \otimes \I)  \\
&  + \w_4^\top (\I \otimes \x \otimes \x \otimes \x) + \w_4^\top ( \x \otimes \I \otimes \x \otimes \x) + \left. \w_4^\top ( \x \otimes \x \otimes \I \otimes  \x) + \w_4^\top ( \x \otimes  \x \otimes \x \otimes \I )+ \cdots \right ). \label{eq:HinftyEnergyFct_deriv} 
\end{align}
We insert the expansion \eqref{eq:HinftyEnergyFct_deriv} into the HJB equation \eqref{eq:HinftyenergyFunctions1_Kron} and subsequently multiply the resulting equation with a factor two. Collecting degree two terms in the resulting equation shows that  $\w_2$ solves 
\begin{align}
\begin{split} \label{eq:W2firsteq}
0 = &  \w_2^\top (\A \x \otimes \x) + \w_2^\top (\x \otimes \A \x) + \text{vec}(\C^\top \C)^\top (\x\otimes \x) \\
 & - \eta \left [ \frac{1}{2} (\w_2^\top (\I \otimes \x) + \w_2^\top (\x \otimes \I)) \right ] \B\B^\top \left [ \frac{1}{2} ((\I \otimes \x^\top)\w_2 + (\x^\top \otimes \I)\w_2) \right ] \\
= &  \left [\w_2^\top (\A \otimes \I) + \w_2^\top (\I \otimes \A) + \text{vec}(\C^\top \C)^\top \right ] (\x\otimes \x) \\
& - \eta \left [ \frac{1}{2} (\w_2^\top (\I \otimes \x) + \w_2^\top (\x \otimes \I)) \right ] \B\B^\top \left [\frac{1}{2} ( (\I \otimes \x^\top)\w_2 + (\x^\top \otimes \I)\w_2 )\right ].
\end{split}
\end{align}
Recall that $\w_2 = \text{vec}(\W_2)$, so that $(\I \otimes \x^\top)\w_2 = \text{vec}(\W_2^\top \x)$ and $(\x^\top \otimes \I)\w_2 = \text{vec}(\W_2 \x)$. We also have that\footnote{\textit{Kronecker-vec relationship}: For two matrices $\A,\B$, we have $(\B^\top \otimes \A) \text{vec}(\X) = \text{vec}(\A \X \B)$.}, $\w_2^\top [\A\otimes \I] = [[\A^\top \otimes \I]\text{vec}(\W_2)]^\top = \text{vec}(\W_2\A)^\top$ and likewise $\w_2^\top [\I\otimes \A] = [(\I \otimes \A^\top)\text{vec}(\W_2)]^\top = \text{vec}(\A^\top \W_2)^\top$. Moreover, for a given matrix $\M$ we have  $\y^\top \M \x = \text{vec}(\M^\top) ^\top (\x \otimes \y)$ so that we can simplify \eqref{eq:W2firsteq} to 
\begin{equation}
0   = \text{vec} \left ( \A^\top \W_2 + \W_2\A + \C^\top \C \right )^\top (\x\otimes \x) - \eta \text{vec} ( \left [\frac{1}{2} (\W_2 + \W_2^\top) \right ] \B\B^\top \left [ \frac{1}{2} (\W_2 + \W_2^\top) \right ] )^\top (\x\otimes \x),
\end{equation}
which can be written compactly in the matrix form as a Riccati equation:
\begin{equation}
\bzero = \A^\top \W_2 + \W_2\A + \C^\top \C  -\eta \left [\frac{1}{2} (\W_2 + \W_2^\top) \right ] \B \B^\top \left [ \frac{1}{2} (\W_2 + \W_2^\top) \right ].
\end{equation}
The solution to this Riccati equation is symmetric, $\W_2 = \W_2^\top$ since the last two terms on the right-hand side are symmetric, namely $\A^\top \W_2 + \W_2 \A = \W_2^\top \A + \A^\top \W_2^\top$. Therefore, 
\begin{equation} \label{eq:riccatiforW2}
\bzero = \A^\top \W_2 + \W_2\A + \C^\top \C - \eta \W_2 \B \B^\top \W_2,
\end{equation}
which is the common $\mathcal{H}_\infty$ Riccati equation from~\eqref{eq:HinftyRiccati2}. Since we assumed that $\gamma>\gamma_0$,~\eqref{eq:riccatiforW2} is uniquely solvable for a positive definite $\W_2$, following Theorem~\ref{thm:HinftyARE}.
Having computed $\w_2$, we proceed as before, i.e.,
 insert \eqref{eq:HinftyEnergyFct_deriv} into \eqref{eq:HinftyenergyFunctions1_Kron} and  multiply the resulting equation with a factor two. Now,
we collect degree three terms to obtain

\vspace{-2ex}
\begin{align}
\nonumber
 0  & =  \w_2^\top (\F(\x\otimes \x) \otimes \x) + \w_2^\top (\x \otimes \F(\x\otimes \x)) + \w_3^\top \left [ (\A \x \otimes \x \otimes \x) +  (\x \otimes \A \x \otimes \x) +(\x \otimes \x \otimes \A \x) \right]   \\
\nonumber
& - \frac{\eta}{4} \left [ \w_2^\top (\I \otimes \x) + \w_2^\top (\x \otimes \I) \right ] \B \B^\top \left [ \w_3^\top \left ( (\I \otimes \x \otimes \x) +  (\x \otimes \I \otimes \x) + (\x \otimes \x \otimes \I)\right ) \right ]^\top \nonumber \\
& - \frac{\eta}{4} \left [ \w_3^\top \left ( (\I \otimes \x \otimes \x) +   (\x \otimes \I \otimes \x) +  (\x \otimes \x \otimes \I) \right ) \right ] \B \B^\top \left [ (\I \otimes \x^\top)\w_2 + (\x^\top \otimes \I)\w_2 \right ] \nonumber\\
\nonumber
& = \left [\w_2^\top \cL_2(\F)  \right ] \kronF{\x}{3} + \left [\w_3^\top \cL_3(\A) \right ] \kronF{\x}{3} \\
& - \frac{\eta}{2} \left [ \w_2^\top (\I \otimes \x) + \w_2^\top (\x \otimes \I) \right ] \B \B^\top \left [ \w_3^\top (\I \otimes \x \otimes \x) +  \w_3^\top (\x \otimes \I \otimes \x) + \w_3^\top (\x \otimes \x \otimes \I) \right ]^\top.
\label{eq:W3_pre}
\end{align}

Next, we seek to compute the unique vector $\w_3$ which satisfies \eqref{eq:W3_pre} as well as the symmetry condition established in Definition \ref{def:sym}. We achieve this goal by first computing a vector $\tw_3 \in \Real^{n^3}$, which satisfies
\begin{align}
\nonumber
 0& = \left [\w_2^\top \cL_2(\F)  \right ] \kronF{\x}{3} + \left [\tw_3^\top \cL_3(\A) \right ] \kronF{\x}{3} \nonumber \\
& - \frac{\eta}{2} \left [ \w_2^\top (\I \otimes \x) + \w_2^\top (\x \otimes \I) \right ] \B \B^\top \left [ \tw_3^\top (\I \otimes \x \otimes \x) +  \tw_3^\top (\x \otimes \I \otimes \x) + \tw_3^\top (\x \otimes \x \otimes \I) \right ]^\top
\label{eq:tW3_pre}
\end{align}
and is non-symmetric in general. Then $\w_3$ is obtained by symmetrizing $\tw_3$.  
First, we simplify \eqref{eq:tW3_pre} by using the fact that $\W_2$ is a symmetric matrix
and write
\begin{equation}
\label{eq:W2B}
   \left [ \w_2^\top (\I \otimes \x) + \w_2^\top (\x \otimes \I) \right ] = \x^\top(\W_2+\W_2^\top) = 2 \x^\top\W_2.
\end{equation}
To rewrite the last line of \eqref{eq:tW3_pre} we introduce perfect shuffle matrices $\S_{n,n^2},\S_{n^2,n} \in \Real^{n^3 \times n^3}$ along the lines of \cite{loan2000TheUbiquitousKroeneckerProduct}, which satisfy
\begin{equation*}
    (\x \otimes \x \otimes \I) = \S_{n,n^2} (\x \otimes \I \otimes \x ) = \S_{n^2,n} (\I \otimes \x \otimes  \x ),
\end{equation*}
as well as $\S_{n,n^2}^\top = \S_{n^2,n}$. The above equation allows us to write
\begin{equation}
    \tw_3^\top (\I \otimes \x \otimes \x) +  \tw_3^\top (\x \otimes \I \otimes \x) + \tw_3^\top (\x \otimes \x \otimes \I) 
    = \tw_3^\top \left[ (\I + \S_{n,n^2} + \S_{n^2,n}) (\x \otimes \x \otimes \I) \right]. \label{eq:w3_sym}
\end{equation}
We introduce
\begin{equation*}
    \tW_3 = \unvec\left[ (\I + \S_{n^2,n} + \S_{n,n^2}) \tw_3 \right] \in \Real^{n \times n^2},
\end{equation*}\
which yields
\begin{equation} 
   \left[ (\I\otimes \x^\top \otimes \x^\top) + (\x^\top \otimes \I \otimes \x^\top) + (\x^\top \otimes \x^\top \otimes \I) \right ] \tw_3 =  \left[ (\x^\top \otimes \x^\top \otimes \I) (\I + \S_{n^2,n} + \S_{n,n^2}) \right]\tw_3 =  \tW_3 (\x\otimes\x).\label{eq:BtW3}
\end{equation}
The last term can be simplified by again resorting to the identity $\y^\top \M \x = \text{vec}(\M^\top)^\top (\x \otimes \y)$. Taken together, equation \eqref{eq:tW3_pre} has the form
\begin{equation}
    \left [\w_2^\top \cL_2(\F)  \right ] \kronF{\x}{3} + \left [\tw_3^\top \cL_3(\A) \right ] \kronF{\x}{3} - \eta \mbox{vec}\left( \tW_3^\top \B\B^\top \W_2 \right)^\top \kronF{\x}{3} = \bzero. 
    \label{eq:W3}
\end{equation}
Expanding the last term once more results in
\begin{equation*}
    \mbox{vec}\left( \tW_3^\top \B\B^\top \W_2 \right)^\top \kronF{\x}{3} ~= \tw_3^\top (\I + \S_{n,n^2} + \S_{n^2,n}) (\x \otimes \x \otimes \B \B^\top\W_2\x) ~= \tw_3^\top \cL_3( \B \B^\top\W_2)\kronF{\x}{3}, 
\end{equation*}
using the arguments similar to \eqref{eq:w3_sym}. 
We can now solve \eqref{eq:W3} for $\tw_3$ via
\begin{equation} \label{eq:nowW3}
  \cL_3((\A - \eta \B\B^\top\W_2)^\top)\tw_3  = -\cL_2({\F^\top})\w_2
\end{equation}
and obtain $\w_3$ by symmetrizing $\tw_3$. 
We repeat this procedure by collecting the degree four terms to obtain
\begin{align}
\nonumber
0 & = \w_3^\top (\F(\x\otimes \x) \otimes \x \otimes \x)   + \w_3^\top (\x \otimes \F(\x\otimes \x) \otimes \x) + \w_3^\top (\x \otimes \x \otimes \F(\x\otimes \x)) \nonumber\\  & + \w_4^\top (\A\x \otimes \x \otimes \x \otimes \x)  + \w_4^\top ( \x \otimes \A\x \otimes \x \otimes \x)   + \w_4^\top ( \x \otimes \x \otimes \A\x \otimes  \x)  + \w_4^\top ( \x \otimes  \x \otimes \x \otimes \A\x ) \nonumber\\
& - \eta \left ( \frac{1}{4}\w_3^\top \left [\I \otimes \x \otimes \x +  \x \otimes \I \otimes \x + \x \otimes \x \otimes \I \right ]  \B\B^\top  [ \I \otimes \x^\top \otimes \x^\top + \x^\top \otimes \I \otimes \x^\top + \x^\top \otimes \x^\top \otimes \I] \w_3  \right ) \nonumber \\
& - 2\eta \left( \frac{1}{4}\w_4^\top [\I \otimes \x \otimes \x \otimes \x + \x \otimes \I \otimes \x \otimes \x + \x \otimes \x \otimes \I \otimes  \x + \x \otimes  \x \otimes \x \otimes \I] \B\B^\top   ((\I \otimes \x^\top + \x^\top \otimes \I)\w_2) \right) \nonumber   \\
& = [\w_3^\top \cL_3(\F) + \tw_4^\top \cL_4 (\A)] \kronF{\x}{4} - \frac{9\eta}{4} \left ( (\x^\top \otimes\x^\top) \W_3^\top \B  \B^\top\W_3 (\x\otimes\x)\right ) - \eta \left ( (\kronF{\x}{3})^\top \tW_4^\top \B  \B^\top \W_2 \x  \right ),
\label{eq:W4_pre}
\end{align}
where the last equality follows from the fact that $\w_3$ is symmetric as well as the definition
\begin{equation*}
    \tW_4 = \unvec\left[ (\I + \S_{n^3,n} + \S_{n^2,n^2} + \S_{n,n^3}) \tw_4 \right] \in \Real^{n \times n^3}.
\end{equation*}
We  now use  $\x^\top \M \x = \text{vec}(\M^\top)^\top (\x \otimes \x)$ again, transpose the linear system \eqref{eq:W4_pre} and equate the coefficients to obtain
\begin{equation}
\bzero = \cL_3(\F^\top) \w_3 + \cL_4 (\A^\top)\tw_4- \frac{9\eta}{4}  \text{vec}(\W_3^\top \B \B^\top \W_3) - \eta \text{vec} (\W_2^\top \B \B^\top \tW_4). 
\end{equation}
Then using 
\begin{equation*}
    {\rm vec}( \tW_4^\top\B\B^\top\W_2)^\top \kronF{\x}{4} = \tw_4^\top (\I + \S_{n^3,n} + \S_{n^2,n^2}  
    + \S_{n,n^3})(\I_{n^3} \otimes  \B \B^\top\W_2) \kronF{\x}{4} = \tw_4^\top \cL_4(\B\B^\top\W_2) \kronF{\x}{4},
\end{equation*}
we have the explicit computation of the degree four terms as a solution to the linear system
\begin{equation}
\cL_4 ((\A - \eta \B \B^\top\W_2)^\top) \tw_4 = - \cL_3(\F^\top) \w_3 + \frac{9\eta}{4} \text{vec}(\W_3^\top \B \B^\top \W_3).
\end{equation}
One can proceed similarly and use induction to show that  for $2<k\leq d$, the coefficient $\tw_k$ can be computed by solving the linear system 
\begin{equation}
  \cL_k((\A - \eta \B \B^\top\W_2)^\top)
  \tw_k = - \cL_{k-1}(\F^\top) \w_{k-1} + \frac{\eta}{4}\sum_{\substack{i,j>2\\ i+j=k+2}} ij~\text{vec}(\W_i^\top \B \B^\top \W_j),
\end{equation}
and $\w_k$ results from symmetrizing $\tw_k$.
\end{proof}
%

\subsection{Past energy function}
Recall that the past energy function $\cE_\gamma^-(\x)$ from \eqref{eq:HinftyenergyFunctions1} is a solution to equation~\eqref{eq:HJB-NLHinfty2}. For quadratic dynamics, i.e., $\ell=2$ in \eqref{eq:FOMx}, equation~\eqref{eq:HJB-NLHinfty2}  becomes 
\begin{equation} 
\label{eq:HinftyenergyFunctions2_Kron}
    0  = \frac{\partial \cE_\gamma^{-}(\x)}{\partial \x} \left [ \A \x + \F(\x\otimes \x) \right ] + \frac{1}{2}  \frac{\partial \cE_\gamma^{-}(\x)}{\partial \x} \B \B^\top \frac{\partial^\top \cE_\gamma^{-}(\x)}{\partial \x} - \frac{\eta}{2}(\text{vec}(\C^\top \C))^\top (\x\otimes \x).
\end{equation}
We follow a similar procedure as in the previous section to compute a polynomial approximation to the energy function  as
\begin{equation} \label{eq:pastenerexp}
    \cE_\gamma^-(\x)  \approx \frac{1}{2}\left ( \v_2^\top \kronF{\x}{2}  + \v_3^\top \kronF{\x}{3} + \cdots \v_d^\top \kronF{\x}{d} \right ).
\end{equation}
where {$\v_k \in \Real^{n^k}$.}
The next theorem illustrates how the coefficients $\{\v_i\}_{i=2}^d$ in
\eqref{eq:pastenerexp} are computed.
\begin{theorem} \label{thm:vi}
Let $\gamma>\gamma_0\geq0$. Let the past energy function $\cE_\gamma^-(\x)$ for the quadratic system ($\ell =2$ and $\F_2 = \F$) in \eqref{eq:FOMx}
be expanded as in \eqref{eq:pastenerexp} with 
   the coefficients $\v_i$ for  $i=2, 3, \ldots, d$.
   Then,  $\v_2 = {\rm{vec}}(\V_2)$ where $\V_2$ is the {symmetric positive definite} solution to the 
   $\mathcal{H}_\infty$ Riccati equation
    \begin{equation}  \label{eq:V2}
        \bzero = \A^\top \V_2 + \V_2\A-\eta \C^\top \C + \V_2 \B \B^\top \V_2.
    \end{equation}
 For $2<k\leq d$,   let let $\tv_k\in \Real^{n^k}$ solve the linear system
    \begin{equation}
        \label{eq:LinSysforvk}
        \cL_k((\A+ \B \B^\top\V_2)^\top)
        \tv_k = - \cL_{k-1}(\F^\top) \v_{k-1} - \frac{1}{4}\sum_{\substack{i,j>2\\ i+j=k+2}} ij~{\normalfont \text{vec}}(\V_i^\top \B \B^\top \V_j).
    \end{equation}
    Then,  the coefficient vector $\v_k = {\rm{vec}}(\V_k) \in \Real^{n^k}$, for $2<k\leq d$, is obtained by the symmetrization of $\tv_k$.
\end{theorem}

\begin{proof} 
The proof follows analogously to that of Theorem~\ref{thm:wi}. 
Therefore, for brevity, we only give a sketch of the proof in this case.   
We start by taking the derivative of the expansion \eqref{eq:pastenerexp}: 
\begin{align}  \label{eq:HinftyEnergyFct2_deriv}
\frac{\partial \cE_\gamma^{-}(\x)}{\partial \x} 
 = & \frac{1}{2}\left (\v_2^\top (\I \otimes \x) + \v_2^\top (\x \otimes \I) + \v_3^\top (\I \otimes \x \otimes \x) +  \v_3^\top (\x \otimes \I \otimes \x) + \v_3^\top (\x \otimes \x \otimes \I) \right.  \nonumber \\
&  \left. + \v_4^\top (\I \otimes \x \otimes \x \otimes \x) + \v_4^\top ( \x \otimes \I \otimes \x \otimes \x) + \v_4^\top ( \x \otimes \x \otimes \I \otimes  \x) + \v_4^\top ( \x \otimes  \x \otimes \x \otimes \I )+ \cdots \right ). 
\end{align}
Following the same steps in the proof of Theorem~\ref{thm:wi}, we insert 
\eqref{eq:HinftyEnergyFct2_deriv} into \eqref{eq:HinftyenergyFunctions2_Kron}
and match the quadratic terms to obtain the Riccati equation
\begin{equation} 
 \bzero = \A^\top \V_2 + \V_2\A  -\eta \C^\top \C + \V_2 \B \B^\top \V_2.
\end{equation}
This last equation does not have the form of the standard $\cH_\infty$ Riccati equation~\eqref{eq:HinftyRiccati1}. Recall the linear case in Theorem~\ref{thm:HinftyARE} where the past energy function is 
$
\cE^{-}_\gamma(\x) = \frac{1}{2} \x^\top \Y_\infty^{-1} \x. 
$
Based on the way we define the polynomial coefficients $\v_i$, and hence the matrix $\V_2$, the quadratic part of the energy function becomes
$
\cE^{-}_\gamma(\x) = \frac{1}{2} \x^\top \V_2 \x. 
$
Therefore, the matrix $\V_2$ that we compute is related to $\Y_\infty$ in~\eqref{eq:HinftyRiccati1} via the relation
$$
 \Y_\infty = \V_2^{-1}.
$$
We now post- and premultiply \eqref{eq:V2} by $\V_2^{-1}$ to obtain
\begin{equation} \label{eq:V2inverse}
 \bzero =  \V_2^{-1} \A^\top  + \A \V_2^{-1} - 
 \eta\V_2^{-1} \C^\top \C\V_2^{-1} + \B \B^\top,
\end{equation}
which is consistent with the usual form of the $\mathcal{H}_\infty$ Riccati equation in~\eqref{eq:HinftyRiccati1}. Following Theorem~\ref{thm:HinftyARE}, for $\gamma>\gamma_0$ we conclude that $\Y_\infty = \V_2^{-1}$ is the stabilizing symmetric positive definite solution to \eqref{eq:V2inverse} and therefore $\V_2$ is a symmetric positive definite solution to \eqref{eq:V2}, which proves the result for $\v_2$.
We continue by matching the cubic terms
\begin{equation}\label{eq:V3}
   \cL_3({\A^\top}+\V_2\B\B^\top)\tv_3  = -\cL_2({\F^\top})\v_2,
\end{equation}
and obtain $\v_3$ by symmetrizing $\tv_3$. Likewise the quartic terms yield
\begin{equation*}
    \cL_4 (\A^\top+\V_2 \B \B^\top)\tv_4 = - \cL_3(\F^\top) \v_3 - \frac{9}{4} \text{vec}(\V_3^\top \B \B^\top \V_3).
\end{equation*}
Then, an induction argument leads to equation~\eqref{eq:LinSysforvk}
for $\tv_k$ for $2<k\leq d$, completing the proof.
\end{proof}

We note that the $\eta$ terms only enter in the $\V_2$ computation, but do not explicitly enter into the formula for the higher-degree $\v_k$ for  $2<k\leq d$. This is to be expected since  \eqref{eq:HinftyenergyFunctions2_Kron} has the $\eta$ term in front of a quadratic polynomial, which enters only the computation of $\V_2$.

\begin{remark}
The theory of nonlinear energy functions and their balancing is formulated locally around the origin, see Section~\ref{sec:NLBal}. Our assumptions produce positive definite solutions to the algebraic Riccati equations, so the polynomial approximations we compute are guaranteed to be positive in a neighborhood of the origin, and may become negative outside that region. When using the polynomial energy functions for balancing and control, one should therefore restrict the region to where the energy function approximations are both positive and well approximated.
\end{remark}
\subsection{Solvability of the Linear Systems}
A natural question that arises from Theorems \ref{thm:wi} and \ref{thm:vi} is when the linear systems of equations \eqref{eq:th6} and \eqref{eq:LinSysforvk} have a unique solution. The following theorem shows that the conditions in the nonlinear setting we consider here are in line with those of linear $\cH_\infty$-balancing. 
\begin{theorem}
Let $\gamma>\gamma_0\geq0$ hold as in  Theorem~\ref{thm:HinftyARE} and $\eta = 1-\gamma^{-2}$ as defined in~\eqref{eq:eta}. Then, for any $k=1,\ldots,d$ the matrices $\cL_k((\A - \eta \B \B^\top\W_2)^\top)$ and $\cL_k((\A + \B \B^\top\V_2)^\top)$ are invertible, thus equations \eqref{eq:th6} and \eqref{eq:LinSysforvk} have unique solutions.
\end{theorem}
\begin{proof}
In this proof we resort to a result from \cite{horn1994topics} that states that for any matrix $\M \in \Real^{n \times n}$ the spectrum of $\cL_k(\M)$, {denoted by $\Lambda\left(\cL_k(\M)\right)$}, is given by all possible sums of $k$ eigenvalues of $\M$, i.e.,
{\begin{equation*}
\Lambda\left(\cL_k(\M)\right) = \left\{\sum_{i \in \mathcal{P}_k} \lambda_i: \lambda_i \in \Lambda(\M) \right\},
\end{equation*}
where $\mathcal{P}_k$ denotes the set of all possible selection of $k$-indices from the set $\{1,2,\ldots,n\}$.}
To prove the first statement we set $\M = (\A - \eta \B \B^\top\W_2)^\top$, where $\W_2$ is the unique stabilizing solution of \eqref{eq:HinftyRiccati2}. This means that all eigenvalues of $\M$ are contained in the open left-half plane. Therefore, all eigenvalues of $\cL_k(\M)$ are contained in the open left-half plane as well, thus $\cL_k(\M)$ is invertible and the $\tw_k$ can be uniquely determined. For the second statement we first rearrange \eqref{eq:HinftyRiccati1} as
\begin{equation*}
 \A^\top + \V_2 \B \B^\top =  - \V_2\A\V_2^{-1} +
 \eta \C^\top \C\V_2^{-1}.
\end{equation*}
We can thus write
\begin{equation*}
 \cL_k((\A + \B \B^\top\V_2)^\top) = \cL_k(- \V_2\A\V_2^{-1} +
 \eta \C^\top \C\V_2^{-1}) = - \kronF{\left[ \V_2 \right]}{k} \cL_k(\A -
 \eta \V_2^{-1} \C^\top \C) \kronF{\left[ \V_2^{-1} \right]}{k}.
\end{equation*}
Considering that $\V_2^{-1}$ is the unique stabilizing positive definite solution to \eqref{eq:HinftyRiccati1}, this shows that all eigenvalues of $\cL_k(\A -
 \eta \V_2^{-1} \C^\top \C)$ are contained in the open left-half plane as well as that $\kronF{\left[ \V_2^{-1} \right]}{k}$ is invertible. This shows that $\cL_k((\A + \B \B^\top\V_2)^\top)$ is the product of invertible matrices, which concludes the proof. 
\end{proof}

\subsection{Complete algorithm}
In Algorithm~\ref{algo:energy_functions}, we provide a complete algorithm for the tensor-based computation of the energy functions assuming a quadratic system ($\ell =2$ and $\F_2 = \F$) in \eqref{eq:FOMx}  and constant input and measurement matrices, summarizing our results from Theorems \ref{thm:wi} and \ref{thm:vi}. 

\begin{algorithm}[H] \label{alg:EnFnctAprx}
	\caption{Computing HJB energy function approximations: $\cE_\gamma^-(\x)$ in \eqref{eq:wi_coeffs} and $\cE_\gamma^+(\x)$ in \eqref{eq:pastenerexp}. }
	\label{algo:energy_functions}
	\begin{algorithmic}[1]
		\Require System matrices $\A, \F, \B, \C$; polynomial degree $d$; constant $\gamma>\gamma_0>0$, $\gamma \neq 1$.
		\Ensure Coefficients $\{\v_i\}_{i=2}^d$ of the past energy and $\{\w_i\}_{i=2}^d$ of the future energy functions.
		\State Set $\eta = (1-\gamma^{-2})$. 
		\State Solve the $\cH_\infty$ Riccati equations
		\begin{align*}
            \bzero & = \A^\top \V_2 + \V_2\A -\eta \C^\top \C + \V_2 \B \B^\top \V_2, \\
            \bzero & = \A^\top \W_2 + \W_2\A + \C^\top \C - \eta \W_2 \B \B^\top \W_2
        \end{align*}
        and set $\v_2 = \text{vec}(\V_2)$ and $\w_2 = \text{vec}(\W_2)$.
        \For{$k=3,4, \ldots, d$}  Solve the systems for $\tv_k$ and $\tw_k$:
        \begin{equation*}
            \cL_k((\A+ \B \B^\top\V_2)^\top)\tv_k = - \cL_{k-1}(\F^\top) \v_{k-1} - \frac{1}{4}\sum_{\substack{i,j>2\\ i+j=k+2}} ij~\text{vec}(\V_i^\top \B \B^\top \V_j)
        \end{equation*}
        and
        \begin{equation*}
            \cL_k((\A - \eta \B \B^\top\W_2)^\top) \tw_k = - \cL_{k-1}(\F^\top) \w_{k-1} + \frac{\eta}{4}\sum_{\substack{i,j>2\\ i+j=k+2}} ij~\text{vec}(\W_i^\top \B \B^\top \W_j)
        \end{equation*}
        Symmetrize $\tw_k$ and $\tv_k$ to obtain $\w_k$ and $\v_k$.
        \EndFor
	\end{algorithmic}
\end{algorithm}

\subsection{Linear Solver and Implementation Details} \label{sec:LinearSolver}
While the earlier sections distilled a structured linear tensor system  for the coefficients of the Taylor-series expansion of the energy functions (see \eqref{eq:th6} and \eqref{eq:LinSysforvk}), the main computational bottleneck lies in solving these structured systems as they grow exponentially in $k$ and polynomially in $n$. We next propose an efficient implementation that heavily leverages the structure of the tensor systems. We present the details for~\eqref{eq:LinSysforvk}. The same ideas apply to~\eqref{eq:th6} as well.

First, to form the right-hand side of~\eqref{eq:LinSysforvk}, we develop a function that efficiently computes products of the form ${\cal L}_{k-1}(\F^\top) \v_{k-1}$.  This is the product of an $n^{k}\times n^{k-1}$ matrix with a vector and a na\"ive approach would require $O(n^{2k-1})$ operations using a level-2 BLAS operation.
However, we exploit the structure of ${\cal L}_k$ in \eqref{eq:dWayLyapunov}, which is the sum of Kronecker products of matrices, where only one term in each sum is not the identity matrix. Thus, using the Kronecker-vec relationship, each term in the sum can be computed by appropriately reshaping $\v_{k-1}$ into an $n^{k-2}\times n$ matrix, performing multiplication by $\F$ on the right, reshaping the result, and accumulating the sum.  This involves (k-1) instances of a ($n^{k-2}\times n$) by ($n\times n^2$) matrix multiplication, which can be performed with level-3 BLAS in  $O(kn^{k+1})$ operations.  For the summation terms on the right-hand side, we can take advantage of the repeated multiplications of $\V_i^\top \B$ and $\W_i^\top \B$, but the cost is dominated by multiplying these stored components together.  The cost of forming the summation terms is $O(kmn^k)$.  In this way, all terms on the right-hand side can be computed efficiently using level-3 BLAS operations with a cost $O(kn^{k+1})+O(kmn^k)$.  Using efficient implementations such as fast matrix multiplication in the BLAS, the actual computational time surpasses these estimates as we see in Section~\ref{sec:Burgers}.

Second, we need to solve linear systems of the form
\begin{equation}
\label{eq:genLinSys}
  \cL_k((\A + \B\B^\top\V_2)^\top)\tv_k = \b.
\end{equation}
We leverage the $k$-way Bartels-Stewart algorithm~\cite{borggaard2021PQR}.
By first performing a Schur factorization of $(\A + \B\B^\top\V_2)^\top = \U \T \U^*$ and defining a matrix $\kronF{\U}{k}=\U\otimes \U\otimes \cdots \otimes \U \in \Real^{n^k\times n^k}$, we convert the linear system \eqref{eq:genLinSys} to 
\begin{displaymath}
  [\kronF{\U}{k}]^* \cL_k((\A + \B\B^\top\V_2)^\top) [\kronF{\U}{k}] \hat{\v}_k = \hat{\b},
\end{displaymath}
where $\hat{\v}_k = [\kronF{\U}{k}]^* \tilde{\v}_k$ and $\hat{\b} = [\kronF{\U}{k}]^*\b$. The multiplications used to compute $\hat{\v}_k$ and $\hat{\b}$ are performed using a recursive algorithm that exploits the Kronecker-vec relationship and also uses level-3 BLAS operations.
The resulting upper triangular linear system is
\begin{equation}
  \cL_k(\T) \hat{\v}_k = \hat{\b},
  \label{eq:TensorSystem}
\end{equation}
{which can be solved by a backsubstitution procedure requiring $n^{k-1}$ dense upper triangular system solutions of size $n$.
Then, we compute the solution $\tv_k = \kronF{\U}{k}\hat{\v}_k$.
The overall computational complexity of solving 
\eqref{eq:genLinSys} in the $k$th step of 
Algorithm~\ref{alg:EnFnctAprx} 
is thus $O(n^{k+1})$ using the classical $O(n^2)$ complexity argument for backsubstitution in dense systems. 
Taken together, the computational complexity of solving the linear systems \eqref{eq:th6} and \eqref{eq:LinSysforvk} for the degree $k$ terms is $O(kmn^k)+O(kn^{k+1})$ or $O(n^{k+1})$ if $n>km$.}

The calculation of coefficients $\{\v_k\}$ and $\{\w_k\}$ are performed using the functions {\tt approxPastEnergy.m} and {\tt approxFutureEnergy.m} in the \texttt{NLbalancing} repository \cite{NLbalancing}.
We remark that linear systems with similar structure to \eqref{eq:genLinSys} have a long history in \textit{linear} systems theory as Lyapunov equations can be written in tensor form, cf.~\cite{brewer1978KroneckerProductsMatrix}. Furthermore, equations of the form ${{\cal L}_k(\S)} \v = \b$ arise from discretization of stationary PDEs with separable coefficients~\cite{chen2019RecursiveBlockedAlgorithms}, as well as when considering polynomial control laws for bilinear~\cite{breiten2018polyFeedback_bilinear} and polynomial systems~\cite{borggaard2019QQR,borggaard2021PQR}.

\section{Numerical Examples}  \label{sec:numerics}
This section illustrates the analysis and computational framework on four numerical examples. We start with two small models to illustrate the theory. First we consider a one-dimensional problem in Section~\ref{sec:exEnergyfcts1} where we can analytically compute the energy functions. Then, we present a two-dimensional example in Section~\ref{sec:exEnergyfcts2} which we solve numerically, where the energy functions can be visualized to build intuition for the higher-dimensional cases. The next two examples test the effectiveness and scalability of the algorithms to compute the energy functions.  These examples are generated from finite element discretizations of PDEs, so we can scale up the state dimension for the numerical study.  Two PDEs are chosen for their quadratic nonlinearities: the Burgers equation (Section~\ref{sec:Burgers}) and the Kuramoto-Sivashinsky equation (Section~\ref{sec:KS}).

\subsection{Illustrative scalar ODE example (n=1)} \label{sec:exEnergyfcts1}
Consider the 1d quadratic system from \cite{bennerGoyal2017BT_quadBilinear}:
\begin{equation}\label{eq:BGexample}
    \dot{x}(t) = a x(t) + \enn x(t)^2 + bu(t), \quad y(t) = c x(t).
\end{equation}
With $\eta = \left ( 1- \gamma^{-2}  \right )$ the HJB equation~\eqref{eq:HJB-NLHinfty2} for the future energy function is the ODE
\begin{equation} 
0 = \frac{\rd \cE_\gamma^{+}(x)}{\rd x} [a x + \enn x^2] - \frac{1}{2} b^2 \eta \left(\frac{\rd \cE_\gamma^{+}(x)}{\rd x}\right )^2+ \frac{1}{2}c^2 x^2.
\end{equation}
Let $p(x) = \frac{\rd \cE_\gamma^{+}(x)}{\rd x}$, then we have a quadratic polynomial in $p(x)$ as
\begin{equation} 
0 =  - \frac{1}{2} b^2 \eta p(x)^2 + [a x + \enn x^2] p(x)  + \frac{1}{2}c^2 x^2,
\end{equation}
for which we obtain the two solutions
\begin{equation}
    p(x)_{1/2} = \frac{ax + \enn x^2 \pm \sqrt{(ax + \enn x^2)^2 +\eta b^2 c^2 x^2}}{\eta b^2} .
\end{equation}
We integrate this solution with respect to $x$ (with condition $\cE_\gamma^{+}(0) =0$) to obtain
\begin{align} \label{eq:ExEgammaplus}
\cE_\gamma^{+}(x) = & \frac{1}{b^2\eta} \left ( \mp \frac{a b^2 c^2 \eta  \sqrt{x^2 ((a + \enn x)^2 + b^2 c^2 \eta)} \log \left (\sqrt{(a + \enn x)^2 + b^2 c^2 \eta} + a + \enn x)\right)}{2 \enn^2 x \sqrt{(a + \enn x)^2 + b^2 c^2 \eta}} \right. \nonumber \\
& \pm \left. \frac{ \sqrt{x^2 ((a + \enn x)^2 + b^2 c^2 \eta)} \left (\frac{(a + \enn x)^2}{3 \enn} - \frac {a (a + \enn x)}{2 \enn} + \frac{b^2 c^2 \eta}{3 \enn}\right )}{\enn x} + \frac{a x^2}{2} + \frac{\enn x^3}{3} \right ).  
\end{align}
Equation~\eqref{eq:HJB-NLHinfty1} for the past energy function becomes
\begin{align} 
0 = \frac{\rd \cE_\gamma^{-}(x)}{\rd x} [a x + \enn x^2] + \frac{1}{2}  b^2\left(\frac{\rd \cE_\gamma^{-}(x)}{\rd x}\right )^2 - \frac{1}{2} \eta c^2 x^2 = \frac{1}{2}b^2q(x)^2 + [ax +\enn x^2]q(x) - \frac{1}{2} \eta c^2 x^2.
\end{align}
Let $q(x) = \frac{\rd \cE_\gamma^{-}}{\rd x}(x)$ which we can solve similarly to get
\begin{align*} 
    q(x)_{1/2} = \frac{-(ax + \enn x^2) \pm \sqrt{(ax + \enn x^2)^2 +\eta b^2 c^2 x^2 }}{b^2}.
\end{align*}
We integrate this solution in $x$ to obtain the past energy function and obtain 
$$
\cE_\gamma^{-}(x) = -\eta \cE_\gamma^{+}(x).
$$
%
%
%
%
We observe that the energy functions are highly nonlinear in $x$. 
For $a=-2$, $\enn=1$, $b=2$, $c=2$, Figure~\ref{fig:Ex1_EgammaPlus} shows $\cE_\gamma^+(x)$ and $\cE_\gamma^-(x)$ for {$-6\leq x \leq 6$} and $\eta=0.5$ ($\gamma=\sqrt{2}$).  
We then apply Algorithm~\ref{alg:EnFnctAprx} with $d=2,4,6,8$ and compute the corresponding polynomial approximations to  $\cE_\gamma^+(x)$ and $\cE_\gamma^-(x)$.
Near the origin, all of the approximations accurately approximate the true energy functions.  However, the quadratic approximation (degree $d=2$) is relatively far from the analytic solution for $|x|>2$ compared to higher degree approximations. Note that degree $d=2$ approximations would be found by either ignoring the quadratic term in \eqref{eq:BGexample}, which is common for linearization approaches, or {other approaches that define localized algebraic Gramians.}  However, over this spatial region, a modest degree $d=8$ approximation is very accurate. 
Going to higher degree approximations extends balanced truncation to a wider region around the origin.  
However, all of these approximations are local and energy function approximations by polynomials might even become negative away from the origin. For example---while not shown on the figure---the fourth and eighth degree polynomial approximations for $\cE_\gamma^{-}$ become negative around $x\approx 8.5$ and $x\approx 6.8$, respectively. This is still in line with the theory of nonlinear balanced truncation, which is always formulated as \textit{local} balancing, see~\cite{scherpen1993balancing,scherpen1994normCoprime,scherpen1996hinfty_balancing}.
In sum, one should test the quality of the energy function approximation before using them for balancing, and one may use their positivity as an indicator of the spatial domain where the model can be locally balanced. 

\begin{figure}
    \centering
    \includegraphics[width=0.49\textwidth]{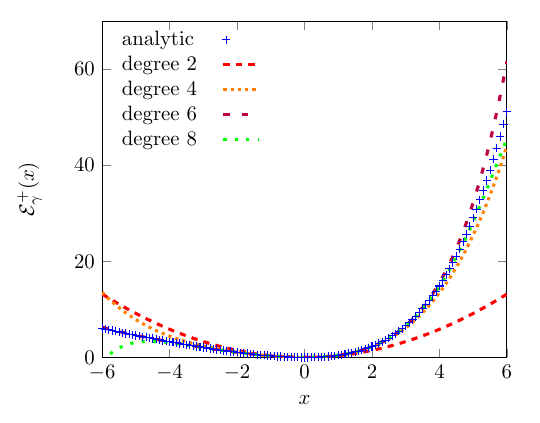}
    \vspace{-.6cm}
    \includegraphics[width=0.49\textwidth]{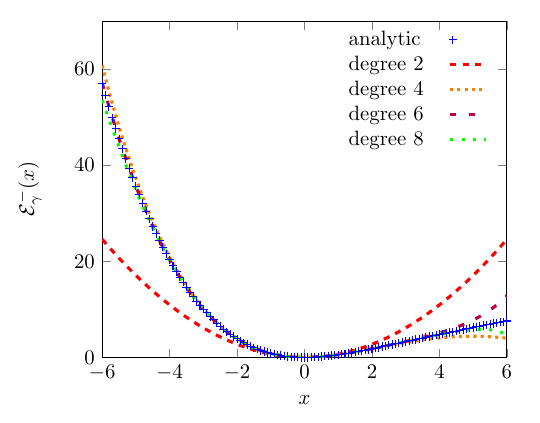}
    \caption{Energy functions $\cE_\gamma^{+}(x)$ (left) and $\cE_\gamma^-(x)$ (right) computed for Example~\ref{sec:exEnergyfcts1}. The analytic solutions are compared to polynomial approximations of varying degree and illustrate that higher-degree approximations are needed to accurately capture the trends in the energy functions.}
    \label{fig:Ex1_EgammaPlus}
\end{figure}
%
%

\subsection{Illustrative ODE system example (n=2)} \label{sec:exEnergyfcts2}
We modify the 2d example from \cite[Sec~IV.C]{kawano2016model} to
\begin{equation}
    \dot{\x} = \A \x + \F (\x \otimes \x) + \B u , \qquad y = \C\x 
\end{equation}
where 
\begin{equation}
\A\x  = \begin{bmatrix} -x_1 + x_2 \\ -x_2 \end{bmatrix}, 
\quad 
\F(\x\otimes \x) = \begin{bmatrix} -x_2^2 \\ 0\end{bmatrix}, 
\quad  
\B = \begin{bmatrix} 1 & 1 \end{bmatrix}^\top, 
\quad 
\C  = \begin{bmatrix} 1 & 1 \end{bmatrix}.
\end{equation}
Figure~\ref{fig:Ex2_Egamma} shows the future and past energy function for $\eta=0.1$ ($\gamma\approx 1.054$) over $-1\leq x_1, x_2 \leq 1$ computed using a degree six approximation in Algorithm~\ref{algo:energy_functions}.
We see that the energy function again shows nonquadratic behavior. For instance, for $\cE_\gamma^+$, consider the diagonal line from $(x_1,x_2) = (-1,-1)$ to $(1,1)$ and note that the contour plots show a larger gradient at the bottom left than at the top right of the figure. This is similar to the asymmetry in Figure~\ref{fig:Ex1_EgammaPlus}.  The same observation (but on the diagonal from $(-1,1)$ to $(1,-1)$) can be made for $\cE_\gamma^-$. A quadratic approximation to the energy function would not be able to match that behavior. 

\begin{figure}[h!]
    \centering
    \includegraphics[width=0.49\textwidth]{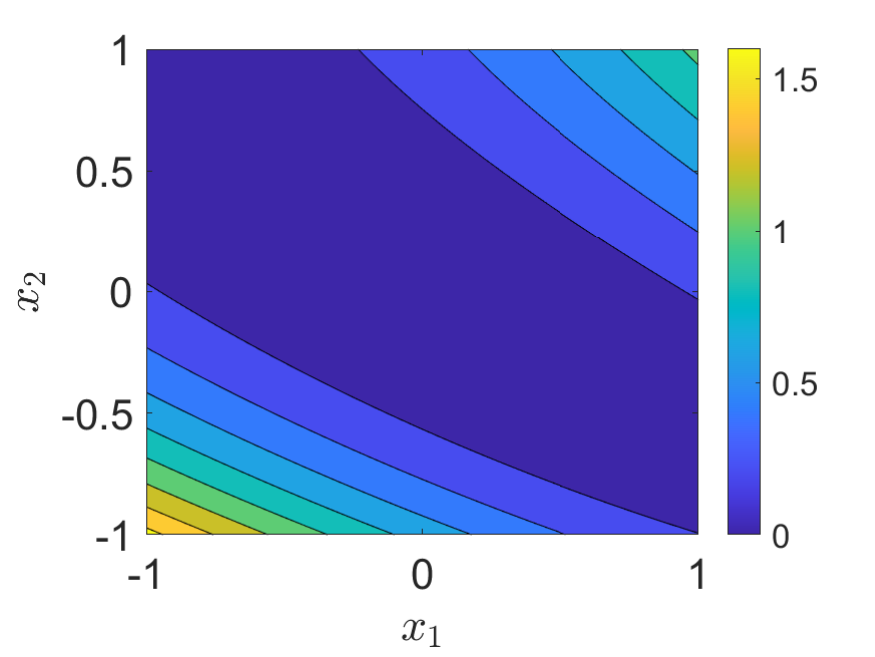}\hfill
    \includegraphics[width=0.49\textwidth]{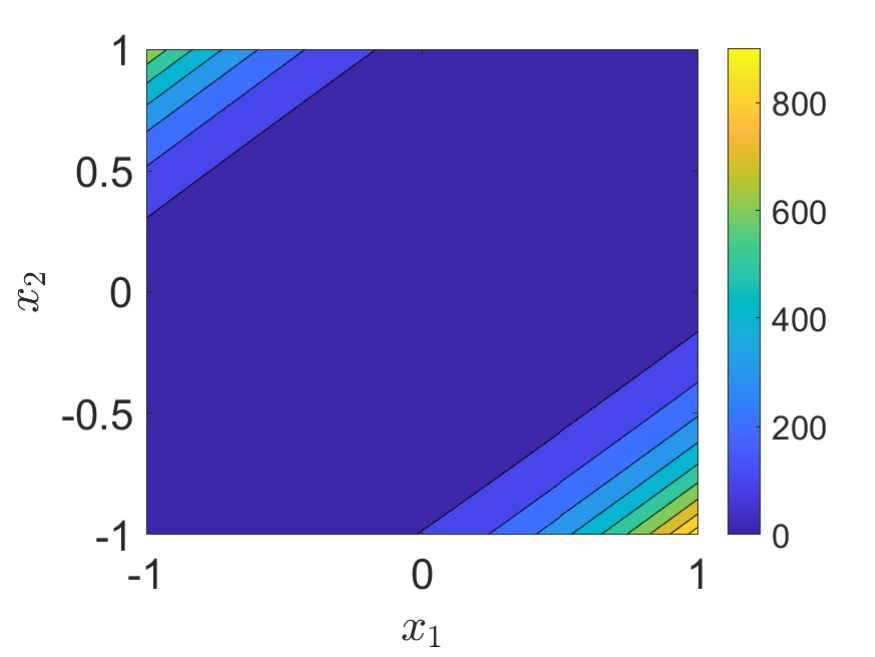}
    \caption{Energy functions $\cE_\gamma^{+}({\bf x})$ (top) and $\cE_\gamma^-({\bf x})$ (bottom) computed for Example~\ref{sec:exEnergyfcts2}.}
    \label{fig:Ex2_Egamma}
\end{figure}

As a verification example, we repeat this experiment for $\eta=0$ and choose the initial condition $x_0=[0.25,-0.25]$.  When $\eta=0$, the future energy function becomes the observability energy function $\cE_c$ from \eqref{eq:energyFunctions2}, which we can calculate analytically by solving first for $x_2(t)$, then using this to calculate $x_1(t)$.  Integrating $y^2(t)=(x_1(t)+x_2(t))^2$ from $0$ to $\infty$ gives us the exact solution for the forward energy function for this special case.   The results in Table~\ref{tab:futureVerification} verify the quality of the polynomial approximations and emphasize the fact that the energy functions are not quadratic (and are exactly quartic for $\eta=0$).  Due to sparsity in $\F$, the coefficients computed using Algorithm~\ref{alg:EnFnctAprx} terminate at the degree 4 term, which is consistent with the analytical solution.

\begin{table}[ht]
  \centering
  \caption{Polynomial approximations of $\cE_\gamma^{+}=\cE_c$ for $\eta=0$ compared to the analytic solution.}
  \label{tab:futureVerification}
  \begin{tabular}{r|c}
  $d$ & ${\cal E}_d^+(\z_0)$ \\
  \hline
    2 & 7.81250000e-03  \\ 
    3 & 9.98263889e-03  \\ 
    4 & 1.01453993e-02  \\ 
  \hline\hline
  analytic & 1.01453993e-02
  \end{tabular}
\end{table}

\begin{remark}
The previous two illustrative examples highlight two important aspects. First, quadratic approximations to the energy function (as done, e.g., in the algebraic Gramian framework~\cite{al1994new,condon2005nonlinear,verriest2006algebraicGramiansNLBT,bennerGoyal2017BT_quadBilinear} and in the LTI case), do not match the nonlinear, asymmetric behavior of the shown energy functions well, except very locally. Thus, there is a need for higher-degree polynomial approximations.
Second, despite the energy functions being highly nonlinear, polynomial approximations capture their behavior well, c.f. Figure~\ref{fig:Ex1_EgammaPlus}. We conclude that polynomial approximations of degree higher than two provide more accurate approximations to the energy functions (qualitatively and quantitatively) over a larger neighborhood of the origin.    
\end{remark}

\subsection{Burgers equation\label{sec:Burgers}}
This test problem has a long history in the study of control for 
distributed parameter systems, e.g.~\cite{thevenet2009nonlinear}, including the 
development of effective computational methods, e.g.~\cite{burns1990control}.   
Consider
\begin{align}
\nonumber
  z_t(x,t) &= \epsilon z_{xx}(x,t) - \frac{1}{2}\left(z^2(x,t)\right)_x +  
  \sum_{j=1}^m b_j^m(x) u_j(t),\\
  \label{eq:averagingObservation}
  y_i(t) &= \int_{\chi_{[(i-1)/p,i/p]}}\hspace{-2em}z(x,t) \text{d} x, \qquad i=1,\ldots,p, 
\end{align}
with initial condition $z(\cdot,0) = z_0(\cdot) \in H_0^1(0,1)$ and the control input defined using the characteristic function $\chi$ as
$b_j^m(x) = \chi_{[(j-1)/m,j/m]}(x)$. The outputs are spatial averages of the solution over equally-spaced subdomains.  We discretize the state equation with $n+1$ linear finite elements leading to $n$ states, set $m=4$ and $p=4$, and chose $\epsilon=0.001$ to make the nonlinearity significant.  
We test the value of the energy functions at
\begin{displaymath}
  z_0(x) = \left\{ \begin{array}{cl}
  0.004 \sin(2\pi x)^2 & x\in(0,0.5) \\
  0 & \mbox{otherwise}
  \end{array} \right. .
\end{displaymath}
The discretized system has the form
\begin{align}\label{eq:discreteCI}
  \widetilde{\bf E}\dot{\bf z} = \widetilde{\bf A}{\bf z} + \widetilde{\bf N}_2\left( {\bf z}\otimes{\bf z} \right) + \widetilde{\bf B}{\bf u},\qquad 
  {\bf y} = \widetilde{\bf C}{\bf z},
\end{align}
where ${\bf z}(t)$ are coefficients of the finite element approximation to $z(x,t)$ and $\z_0$ are finite element coefficients from a best approximation to $z_0$.  To place this in the form \eqref{eq:FOMx}-\eqref{eq:FOMy} with $\ell=2$, we introduce
the change of variables ${\bf x} = {\bf S}{\bf z}$ where ${\bf S} = \widetilde{\bf E}^{1/2}$ (a matrix square root of the finite element mass matrix).  Then defining
$\A=\S^{-1}\widetilde{\A}{\S}^{-1}$, $\B={\bf S}^{-1}\widetilde{\bf B}$, ${\bf C} = \widetilde{\bf C}{\bf S}^{-1}$, $\widetilde{\bf N}_2 = {\bf N}_2({\bf S}^{-1}\otimes{\bf S}^{-1})$ leads to a system in the required form.

In Table~\ref{tab:Burgers_CPU}, we compute cubic approximations to the future energy function for increasing discretization sizes of the Burgers equation using the value $\eta=0.9$.  
The table illustrates the efficiency of the our algorithm for increasing problem sizes. 
Computations were performed on a 2019 Mac Pro desktop with 2.7 GHz 24-core Intel Xeon W processors.  While our flop-count analysis in Section~\ref{sec:LinearSolver} predicts computational cost with growth of $O(n^4)$ (since $d=3$), the table indicates a CPU time scaling of approximately $O(n^{2.84})$.  Repeating this for quartic approximations (Table~\ref{tab:Burgers_CPU4}), we find growth of $O(n^{3.57})$. This suggests that CPU time scales more like $O(n^{d})$ for our problem sizes.
Notably, our proposed framework allows us to compute a high-resolution approximation of the cubic term in the energy function for a state-space dimension of $n=1024$ using 64GB of RAM. The cubic coefficient is defined through the solution of linear system of size $10^9$.  By exploiting the Kronecker structure of the linear system and using the efficient BLAS-3 level implementation outlined in Section~\ref{sec:LinearSolver}, we can perform this calculation just over 2 hours (7,930s) of CPU time.

\begin{table}[ht]
    \centering
    \caption{Degree 3 Future Energy Function Approximation ($d=3$) for the Discretized Burgers Equation.}
    \label{tab:Burgers_CPU}
    \begin{tabular}{r|ccc}
    $n$ & $n^{3}$ & CPU sec & ${\cal E}_3^+(\z_0)$ \\
    \hline
         8 & 5.1200e+02 & 2.96e-02 & 1.144557e-06 \\ 
        16 & 4.0960e+03 & 1.08e-02 & 1.116244e-06 \\ 
        32 & 3.2768e+04 & 5.96e-02 & 1.093503e-06 \\ 
        64 & 2.6214e+05 & 4.40e-01 & 1.099870e-06 \\ 
       128 & 2.0972e+06 & 4.29e+00 & 1.097715e-06 \\ 
       256 & 1.6777e+07 & 5.48e+01 & 1.095300e-06 \\ 
       512 & 1.3422e+08 & 6.63e+02 & 1.096322e-06 \\ 
      1024 & 1.0737e+09 & 7.93e+03 & 1.096093e-06    
    \end{tabular}
\end{table}

\begin{table}[ht]
  \centering
  \caption{Degree 4 Future Energy Function Approximation ($d=4$) for the Discretized 
  Burgers equation.}
  \label{tab:Burgers_CPU4}
  \begin{tabular}{r|ccc}
  $n$ & $n^{4}$ & CPU sec & ${\cal E}_4^+(\z_0)$ \\
  \hline
      8 & 4.0960e+03 & 4.49e-02 & 1.144783e-06 \\ 
     16 & 6.5536e+04 & 1.38e-01 & 1.116636e-06 \\ 
     32 & 1.0486e+06 & 1.77e+00 & 1.093928e-06 \\ 
     64 & 1.6777e+07 & 3.28e+01 & 1.100306e-06 \\ 
    128 & 2.6844e+08 & 6.86e+02 & 1.098153e-06
  \end{tabular}
\end{table}

For $n=8$, Table~\ref{tab:Burgers_energy} shows the energy functions ${\cal E}_d^-(\z_0)$
and ${\cal E}_d^+(\z_0)$ as the polynomial degree $d$ increases. The past energy function
${\cal E}_d^-(\z_0)$ converges more slowly in this specific example, while the future energy function ${\cal E}_d^+(\z_0)$ converges quickly with~$d$.

\begin{table}[ht]
  \centering
  \caption{Energy Function Approximations  with Increasing Degree for the Discretized Burgers Equation ($n=8$).}
  \label{tab:Burgers_energy}
  \begin{tabular}{r|cc}
  $d$ & ${\cal E}_d^-(\z_0)$ & ${\cal E}_d^+(\z_0)$ \\
  \hline
    2 & 3.161325e-05 & 1.146135e-06 \\ 
    3 & 2.731740e-05 & 1.144557e-06 \\ 
    4 & 2.370917e-05 & 1.144783e-06 \\ 
    5 & 2.593642e-05 & 1.144792e-06 \\ 
    6 & 2.662942e-05 & 1.144791e-06 \\ 
    7 & 2.519892e-05 & 1.144791e-06 \\ 
    8 & 2.538956e-05 & 1.144791e-06 
  \end{tabular}
\end{table}

\subsection{Kuramoto-Sivashinsky equation} \label{sec:KS}

We consider the  system described by the Kuramoto-Sivashinsky equation 
for $x\in(0,1)$ and $t>0$,
\begin{align}\label{eq:ks}
  z_t(x,t) = & -\epsilon z_{xx}(x,t) -\epsilon^2 z_{xxxx}(x,t) \nonumber \\
   &  -\epsilon (z(x,t)^2)_x + \sum_{j=1}^{m} b_j^m(x)u_j(t)
\end{align}
subject to periodic boundary conditions $ z(0,t) = z(1,t)$ and $z_x(0,t) = z_x(1,t)$.  We use the same control input functions $b_j^m$ as in Section~\ref{sec:Burgers} as well as the integral observation operator in \eqref{eq:averagingObservation}.
This version of the equations can be found in \cite{dankowiczLocalModelsSpatiotemporally1996,holmes_lumley_berkooz_1996}. Here we discretize the control problem using Hermite cubic finite elements to account for the fourth derivative term.  The same change of variables as described in the previous example brings this problem to the standard form (\eqref{eq:FOMx}-\eqref{eq:FOMy}).  We use an initial condition of $z(x,0)=z_0(x)=\frac{0.1}{\sqrt{\epsilon}}\sin(4\pi x)$ and the parameter $\epsilon=1/13.0291^2$, which is known to exhibit heteroclinic cycles in the open-loop system.  
Repeating the experiments for the Burgers equation, with $m=5$ and $p=2$, and here choosing $\eta=0.1$, we find a similar convergence for the future energy function approximation in Table~\ref{tab:KS_CPU}. Likewise, with $n=16$, we have the convergence of the past and future energy functions with increasing polynomial degree seen in Table~\ref{tab:KS_energy}.  The energy function approximations converge at a point near the origin in both number of states $n$ and polynomial degree $d$.  We also observe a growth in CPU time that outperforms the predicted flop-count in Section~\ref{sec:LinearSolver}.

\begin{table}[ht]
    \centering
    \caption{Future Energy Function Approximation ($d=3$) for the Discretized Kuramoto-Sivashinsky Equation.}
    \label{tab:KS_CPU}
    \begin{tabular}{r | c  c  c}
    $n$ & $n^{3}$ & CPU sec & ${\cal E}_3^+(\z_0)$ \\
    \hline
        16 & 4.0960e+03 & 1.20e-02 & 4.369195e+00 \\ 
        32 & 3.2768e+04 & 8.44e-02 & 5.099752e+00 \\ 
        64 & 2.6214e+05 & 5.54e-01 & 4.793412e+00 \\ 
       128 & 2.0972e+06 & 9.14e+00 & 4.732940e+00 \\ 
       256 & 1.6777e+07 & 1.37e+02 & 4.811878e+00 \\ 
       512 & 1.3422e+08 & 1.70e+03 & 4.827930e+00 \\ 
      1024 & 1.0737e+09 & 2.04e+04 & 4.807904e+00
    \end{tabular}
\end{table}

\begin{table}[ht]
  \centering
    \caption{Approximating Energy Functions with Increasing Degree with $n=16$ for the Discretized Kuramoto-Sivashinsky Equation.}
  \label{tab:KS_energy}
  \begin{tabular}{r|cc}
  $d$ & ${\cal E}_d^-(\z_0)$ (CPU sec) & ${\cal E}_d^+(\z_0)$ (CPU sec)\\
  \hline
    2 & 5.2913043e+00 (2.71e-03) & 4.3690773e+00 (6.81e-03) \\ 
    3 & 4.1573639e+00 (1.47e-02) & 4.3691951e+00 (9.88e-03) \\ 
    4 & 4.2904579e+00 (2.22e-01) & 4.3469410e+00 (1.37e-01) \\ 
    5 & 4.2814109e+00 (4.00e+00) & 4.3467633e+00 (2.40e+00) \\ 
    6 & 4.2830236e+00 (7.46e+01) & 4.3467610e+00 (4.39e+01)
  \end{tabular}
\end{table}

\section{Conclusions and future directions} \label{sec:conclusions}
We proposed a unifying and scalable approach to computing a family of $\cH_\infty$ energy functions. We employed Taylor series expansion for solving a class of parametrized HJB equations, which are at the core of nonlinear balanced truncation. 
The proposed approach was made feasible through deriving a linear tensor structure for the coefficients of the Taylor series, and by heavily exploiting that structure in the resulting linear systems with billions of unknowns.
Our last numerical example considered a classical control problem for semi-discretized PDEs in moderate to large dimensions where we computed high-fidelity polynomial approximations to the $\cH_\infty$ energy functions.  This computational framework will pave the way to perform nonlinear balanced truncation using the true energy functions (without quadratic approximation and/or without algebraic Gramians) for systems  with moderately large dimensions. 
In future work, we plan to investigate further the convergence of the polynomial approximations to the energy functions and extend the results to control-affine systems with polynomial input, output and drift terms. 
A further benefit of this approach is that the energy functions in the $\mathcal{H}_\infty$ and HJB balancing methods automatically produce controllers for nonlinear systems, see Section~\ref{sec:NLBal}. In \cite{breiten2018polyFeedback_bilinear,almubarak2019infinite,borggaard2019QQR,borggaard2021PQR} it is shown that polynomial control laws have better performance than linear quadratic regulators.

\section{Acknowledgements}
We thank Nick Corbin for valuable comments on the several drafts of this manuscript. This work was supported in part by the NSF under Grant CMMI-2130727 and is based upon work supported by the National Science Foundation under Grant No. DMS-1929284 while the first three authors were in residence at the Institute for Computational and Experimental Research in Mathematics in Providence, RI, during the Spring 2020 Semester Program "Model and dimension reduction in uncertain and dynamic systems" and Spring 2020 Reunion Event. 

\bibliographystyle{elsarticle-num} 
\bibliography{NLbal}

\begin{appendices}
\section{HJB balancing energy functions} \label{ss:HJBbal}
The Hamilton-Jacobi-Bellman (HJB) balancing approach developed in~\cite{scherpen1994normCoprime}  applies to \textit{unstable} nonlinear systems as well, and hence defines energy functions that include both the control and observation penalties simultaneously. For linear systems, HJB-balancing and LQG-balancing~\cite{verriest1981suboptimal,jonckheere1983LQGbalancing} are identical concepts. For nonlinear systems they lead to different definitions. This section focuses on HJB balancing, which defines the \textit{past energy function} ($\cE^{-}(\x_0)$) and \textit{future energy function} ($\cE^{+}(\x_0)$) for a nonlinear system as
\begin{equation} \label{eq:HJBenergyFunctions1}
\cE^{-}(\x_0)  :=\min_{\substack{\u \in L_{2}(-\infty, 0] \\ \x(-\infty) = \bzero \\ \x(0) = \x_0}} \ \frac{1}{2} \int_{-\infty}^{0} \Vert \y(t) \Vert^2  +  \Vert \u(t) \Vert^2 \text{d}t,
\end{equation}
\begin{equation}\label{eq:HJBenergyFunctions2}
\cE^{+}(\x_0) := \min_{\substack{\u \in L_{2}[0,\infty) \\ \x(0) = \x_0 \\ \x(\infty) = \bzero}} \frac{1}{2} \int_{0}^{\infty} \Vert \y(t) \Vert^2 + \Vert \u(t) \Vert^2 \text{d}t. 
\end{equation}
As shown in~\cite[Thm. 15]{scherpen1994normCoprime}, the past energy function $\cE^{-}(\x_0)$ is a solution to the Hamilton-Jacobi-Bellman equation
\begin{equation} \label{eq:HJB-NLLQG4}
0 = \frac{\partial \cE^{-}(\x)}{\partial \x} \f(\x) + \frac{1}{2}\frac{\partial \cE^{-}(\x)}{\partial \x} \g(\x) \g(\x)^\top \frac{\partial^\top \cE^{-}(\x)}{\partial \x} - \frac{1}{2}\h(\x)^\top \h(\x),
\end{equation}
and the future energy function $\cE^{+}(\x_0)$ is a solution to the Hamilton-Jacobi-Bellman equation
\begin{equation} \label{eq:HJB-NLLQG3}
0  = \frac{\partial \cE^{+}(\x)}{\partial \x} \f(\x) - \frac{1}{2}\frac{\partial \cE^{+}(\x)}{\partial \x} \g(\x) \g(\x)^\top \frac{\partial^\top \cE^{+}(\x)}{\partial \x} + \frac{1}{2}\h(\x)^\top \h(\x). 
\end{equation}

\section{LQG balancing energy functions} \label{ss:LQGbal}
While HJB balancing and LQG balancing are identical for linear systems, their extensions to nonlinear systems are different. The authors in \cite{scherpen1994normCoprime} consider an extension of the LQG balancing problem to nonlinear systems, noting that \textit{``the formulation of LQG balancing for linear systems cannot easily be extended to nonlinear systems. The usual stochastic formulation of the LQG problem seems not to be the right formulation for nonlinear systems. However, there exists a deterministic formulation of the LQG problem, which is equivalent to the stochastic formulation, and which has been extended to nonlinear systems; see \cite{mortensen1968maximum} and \cite{hijab1979minimum}".}
The deterministic nonlinear extension of LQG balancing poses the optimization
\begin{align}
\begin{split}
	\min_\u	 \quad &  w_0(\x_0) + \frac{1}{2}  \int_0^{t_f} \Vert \u(t) \Vert^2 + \Vert \boldsymbol{\eta}(t) \Vert^2 \ \text{d} t \\
	\text{s.t.} &\quad 	  \dot{\x}(t) = \f(\x(t)) + \g(\x) \u(t) \\
								& \quad \y(t) = \h (\x(t))  + \boldsymbol{\eta}(t),
\end{split}								
\end{align}
where $\boldsymbol{\eta}$ is an additive noise, $t_f$ is a final time, and $w_0(\x_0)$ is a real-valued function representing the initial cost, with $w_0(\bzero) =0$, see~\cite{hijab1979minimum}.
Following \cite[Sec. 5.1]{scherpen1994normCoprime}, let $w(t,\x)$ be a solution to the \textit{Mortensen equation}
\begin{equation} \label{eq:HJB-NLLQG1}
0 = \frac{\partial w(t,\x)}{\partial t} + \h(\x) \y(t) + \frac{\partial w(t,\x)}{\partial \x} \f(\x) 
+\frac{1}{2}\frac{\partial w(t,\x)}{\partial \x} \g(\x) \g(\x)^\top \frac{\partial^\top w(t,\x)}{\partial \x} - \frac{1}{2}\h(\x)^\top \h(\x),
\end{equation}
with initial condition $w(0,\x) = w_0(\x)$. Also let $v(\x)$ be the smooth positive solution to the HJB equation
\begin{equation} \label{eq:HJB-NLLQG2}
0 =  \frac{\partial v(\x)}{\partial \x} \f(\x) - \frac{1}{2}\frac{\partial v(\x)}{\partial \x} \g(\x) \g(\x)^\top \frac{\partial v(\x)^\top}{\partial \x} + \frac{1}{2}\h(\x)^\top  \h(\x),
\end{equation}	
with $v(\bzero) = 0$.
Then the dynamics of the deterministic estimate $\hat{\x}$ of the state $\x$ are given by
\begin{equation}
\dot{\hat{\x}} = \f(\hat{\x}) + \g(\hat{\x}) \u + \left ( \frac{\partial^2 w}{\partial \x^2} (t,\hat{\x}) \right )^{-1}  \left ( \frac{\partial \h}{\partial \x}(\hat{\x}) \right ) ^\top (\y(t) - \h(\hat{\x})),
\end{equation}	
and a control can be obtained via 
\begin{equation}\label{eq:control}
\u(\x) = - \g(\hat{\x})^\top \frac{\partial^\top v}{\partial \x}(\hat{\x}).
\end{equation}
The energy functions $w(t,\x)$ and $v(\x)$ are essential in obtaining the control and filter for the system, yet the time-dependence of $w(t,\x)$ causes difficulties. The HJB past energy function~\eqref{eq:HJB-NLLQG4} can be viewed as a steady state version of~\eqref{eq:HJB-NLLQG1}. 
Moreover, in the case of an LTI system, the control $\u(t)$ and observer $\hat{\x}(t)$ are the usual LQG solutions (see 
\cite[Prop. 3.4]{glover91Hinftybalancing}), as $\frac{\partial^2 w}{\partial \x^2} (t,\hat{\x}) \equiv \Q$ is the solution of the filter algebraic Riccati equation (ARE) and $\frac{\partial v}{\partial \x}^\top (t,\hat{\x}) \equiv \P\x$, where $\P$ is the solution of the control ARE.
\end{appendices}

\end{document}